\documentclass[11pt,oneside,reqno]{amsart}
\usepackage{amssymb}
\usepackage{}
\usepackage{amsmath}
\usepackage{graphicx}
\usepackage{mathrsfs}
\usepackage{amsfonts}
\usepackage{enumerate,amsmath,amssymb,amsthm}
\usepackage{hyperref}
\hypersetup{backref,
colorlinks=true,
linkcolor=blue,
anchorcolor=blue,
citecolor=blue}

\newcommand\dela[1]{{\color{red}}}

\numberwithin{equation}{section}

\newcommand{\be}{\begin{eqnarray}}
\newcommand{\ee}{\end{eqnarray}}
\newcommand{\ce}{\begin{eqnarray*}}
\newcommand{\de}{\end{eqnarray*}}
\newtheorem{theorem}{Theorem}[section]
\newtheorem{lemma}[theorem]{Lemma}
\newtheorem{claim}[theorem]{Claim}
\newtheorem{remark}[theorem]{Remark}
\newtheorem{definition}[theorem]{Definition}
\newtheorem{proposition}[theorem]{Proposition}
\newtheorem{Examples}[theorem]{Example}
\newtheorem{corollary}[theorem]{Corollary}
\newtheorem{assumption}{Assumption}[section]

\def\[{{\Big[}}
\def\]{{\Big]}}
\def\<{{\langle}}
\def\>{{\rangle}}
\def\({{\Big(}}
\def\){{\Big)}}

\def\bx{{\mathbf{x}}}

\def\dif{{\mathord{{\rm d}}}}

\def\={&\!\!=\!\!&}
\def\bt{\begin{theorem}}
\def\et{\end{theorem}}
\def\bt{\begin{claim}}
\def\et{\end{claim}}
\def\bl{\begin{lemma}}
\def\el{\end{lemma}}
\def\br{\begin{remark}}
\def\er{\end{remark}}
\def\bas{\begin{assumption}}
\def\eas{\end{assumption}}
\def\bd{\begin{definition}}
\def\ed{\end{definition}}
\def\bp{\begin{proposition}}
\def\ep{\end{proposition}}
\def\bc{\begin{corollary}}
\def\ec{\end{corollary}}
\def\bx{\begin{Examples}}
\def\ex{\end{Examples}}

\def\cB{{\mathcal B}}

\def\cF{{\mathcal F}}

\def\mN{{\mathbb N}}

\def\geq{\geqslant}
\def\leq{\leqslant}

\allowdisplaybreaks
\usepackage{color}

\title[Large deviations for locally monotone stochastic partial differential equations driven by L\'evy noise]{\bf{Large deviations for locally monotone SPDEs driven by L\'evy noise}}

\author{Weina Wu}
\address{Weina Wu, School of Economics, Nanjing University of Finance and Economics, Nanjing, Jiangsu 210023, China; Faculty of Mathematics, University of Bielefeld, D-33615 Bielefeld, Germany.}
\thanks{Weina Wu is supported by the National Natural Science Foundation of China (NSFC) (No.11901285), China Scholarship Council (CSC) (No.202008320239), Deutsche Forschungsgemeinschaft (DFG, German Research Foundation) through CRC 1283}
\email{wuweinaforever@163.com}

\author{Jianliang Zhai}
\address{Jianliang Zhai, School of Mathematical Sciences,
University of Science and Technology of China, Hefei, Anhui 230026, China.}
\thanks{Jianliang Zhai is supported by NSFC (No.12131019, 11971456, 11721101), and the Fundamental Research Funds for the Central Universities (No. WK3470000031, No. WK3470000016).}
\email{zhaijl@ustc.edu.cn}

\author{Jiahui Zhu}
\address{Jiahui Zhu, School of Science, Zhejiang University of Technology, Hangzhou 310019, China.}
\thanks{Jiahui Zhu is supported by NSFC (No.12071433).}
\email{jiahuizhu@zjut.edu.cn}
\date{\today}
\begin{document}

\maketitle

\begin{abstract}
We establish a Freidlin-Wentzell type large deviation principle (LDP) for a class of stochastic partial differential equations with
locally monotone coefficients driven by L\'evy noise. Our results essentially improve a recent work on this topic (Bernoulli, 2018) by the second named author of this paper and his collaborator, because we drop the compactness embedding assumptions, and we also make the conditions for the coefficient of the noise term more specific and weaker. To obtain our results, we utilize an improved sufficient criteria of Budhiraja, Chen, Dupuis,
and Maroulas for functions of Poisson random measures, and the techniques introduced by the first and second named authors
of this paper in \cite{WZSIAM}
play important roles.

As an application, for the first time, the Freidlin-Wentzell type LDPs for many SPDEs driven by L\'evy noise in unbounded domains of $\mathbb{R}^d$, which are generally lack of compactness embeddings properties, are achieved, like e.g., stochastic $p$-Laplace equation,  stochastic Burgers-type equations, stochastic 2D Navier-Stokes equations, stochastic equations of non-Newtonian fluids, etc.

\vspace{3mm} {\bf Keywords:} Freidlin-Wentzell type large deviation principle; L\'evy noise; locally monotone; stochastic partial differential equations
\end{abstract}


\section{Introduction}
The Freidlin-Wentzell type large deviation principle (LDP) describes how the solutions of equations behave when the noise in the equation approaches zero, see \cite{FW} for more details. The LDP theory provides a powerful mathematical framework for studying the behavior of rare events and extreme fluctuations in a wide variety of complex systems. For example, large deviation theory plays a crucial role in statistical mechanics by providing a framework for analyzing rare fluctuations in physical systems, shedding light on phenomena such as phase transitions and fluctuations in thermodynamic quantities (see e.g. \cite{Ellis} for more details); in finance and economics, large deviation theory is applied to model extreme events in financial markets, such as stock market crashes or large price movements (see e.g. \cite{FGGJT} for more details).

Different from the Gaussian noise, which is commonly used in many standard engineering and statistical applications, where the underlying processes can be adequately described by normal distributions, L\'evy noise provides a flexible and powerful framework for modeling complex systems that exhibit non-Gaussian behavior, heavy-tailed distributions, and long-range dependence. This makes it well-suited for capturing the dynamics of real-world phenomena that cannot be adequately described by traditional Gaussian processes, see e.g. \cite{PZ} for more details about L\'evy noise.

Although there have been lots of papers investigated the LDPs for stochastic evolution equations (SEEs) and stochastic partial differential equations (SPDEs) driven by Gaussian noise (see e.g. \cite{BDMAOP, CW, CR, DM, FG, MSZ, WZJDE} and the references therein), there has not been very much results about the LDPs for SEEs and SPDEs driven by L\'evy noise. The first paper about large deviations for SEEs driven by L\'evy processes is \cite{RZ}, in which the authors obtained large deviation results for an Ornstein-Uhlenbeck type process driven by additive L\'evy noise. By using the contraction principle in the theory of large deviations (see e.g. \cite[Theorem 4.2]{DZ}), \cite{XuZhang} dealt with the LDPs for two-dimensional stochastic Navier-Stokes
equations driven by additive L\'evy noise. In \cite{SZ}, the authors used the theorem of Varadhan and Bryc (see e.g. \cite[Theorem 1.3.8]{DE})  coupled with
the techniques of Feng and Kurtz \cite{FK} to prove a LDP result for solutions of abstract SEEs perturbed by multiplicative L\'evy noise, on a larger space (hence with a weaker topology) than the actual state space of the solution. Since Budhiraja, Dupuis and Maroulas  introduced the weak convergence
approach in \cite{Budhiraja-Dupuis-Maroulas.}, it has been applied to study the LDPs in various dynamical systems driven by L\'evy noise, like e.g., the LDPs for \eqref{eq SPDE 01} on some nuclear spaces with a monotone condition and Hilbert spaces with a locally monotone condition were established in \cite{BCDSPA} and \cite{XZ}, respectively; in \cite{YZZ}, the authors proved the LDPs for a class of SEEs driven by multiplicative L\'evy noise on Hilbert spaces, which is  the actual state space of the solution; the LDPs for two-dimensional stochastic Navier-Stokes equations driven by multiplicative L\'evy noises were investigated in \cite{ZZ} and \cite{BPZJEMS} (we remark here that the strong solutions in \cite{XuZhang, ZZ} are concerned in the probabilistic sense, while the strong solutions in \cite{BPZJEMS} are concerned in
the PDE sense). We stress that all the existing results on this topic (though we do not list all the works here), except for two papers \cite{WZSIAM} and \cite{ZLZ} (cf. the next paragraph for further details), require the compactness of embeddings. However, it is very common to encounter situations where embeddings lack compactness, especially in some interesting state spaces such as the general $\sigma$-finite measurable spaces, and in particular, the unbounded domains in $\Bbb{R}^d$. Therefore, without the compactness conditions, whether the LDPs for SPDEs driven by L\'evy noise can be established or not is significantly worthwhile to be investigated.

As far as we know, there are only two papers about the LDPs for SPDEs driven by L\'evy noises without using the compactness assumptions. In \cite{WZSIAM}, the first and second named authors of this paper proved the LDPs for a class of stochastic porous media equations driven by L\'evy noise on general $\sigma$-finite measurable spaces. In \cite{ZLZ}, the second and third named authors of this paper and Liu established a LDP for stochastic nonlinear Schr\"odinger equation with either focusing or defocusing nonlinearity driven
by nonlinear multiplicative L\'evy noise in the Marcus canonical form on the whole  $\Bbb{R}^d$. However, there are many other SPDEs, like e.g., stochastic $p$-Laplace equation, stochastic Burgers-type equations, stochastic 2D Navier-Stokes equations, stochastic equations of non-Newtonian fluids, etc., and the LDPs for these SPDEs driven by L\'evy noises in the absence of compactness embeddings are open problems to be solved. Since these equations can be formulated within the framework of locally monotone SPDEs, this strongly motivates us to solve the problems by establishing the LDPs for a class of SPDEs with locally monotone coefficients driven by L\'evy noises, without using the compactness embeddings hypothesis (cf. Theorem \ref{th main LDP} for our main result). We would like to emphasize that, as an application of our current work, the LDPs for the aforementioned SPDE examples driven by L\'evy noises in unbounded domains of $\mathbb{R}^d$, are achieved for the first time.

We will employ the same criteria as the ones used in \cite{WZSIAM, ZLZ}. It is an improved sufficient criteria proposed in \cite{LSZZ} by the second named author of this paper and his collaborators, and can be regarded as an adaption of the weak convergence
approach introduced in \cite{Budhiraja-Dupuis-Maroulas.} and \cite{BCDSPA} for the case of Poisson random measures. The improved sufficient criteria was first introduced by Matoussi, Sabbagh, and Zhang in \cite{MSZ} for the Wiener case, and has been proved to be more effective and suitable to deal with SPDEs with highly nonlinear terms, see e.g. \cite{DWZZ, WZJDE} for the Wiener case. The main procedures of our proofs are to establish the well-posedness of solutions to the skeleton equations, the convergence of the solutions to the skeleton equations, and some convergence between the solutions to the controlled SPDEs and the solutions to the skeleton equations.

As mentioned before, in \cite{XZ}, the second named author of this paper and Xiong also proved the LDPs for \eqref{eq SPDE 01}, but our current paper essentially improves the results in \cite{XZ}.
 The main improvement is that we drop the compactness assumption for the Gelfand triples used in \cite{XZ}.
Another improvement is that we  make the conditions for the coefficient of the noise term more specific and weaker.
To be precise,
we note that $L_{2}(\nu_T)\cap \mathcal{H}_2\subset L_{\beta+2}(\nu_T)\cap \mathcal{H}_2$, $\beta\geq0$, (cf. Section \ref{Hypotheses and main result} for the definitions of the spaces and Remark \ref{H6 remark} for the proof; see also \cite{BD book}), which makes \cite[page:2848, \textbf{(H5)}]{XZ} more concrete. In our paper, we use the concrete assumption (cf. \textbf{(H6)} in Section \ref{Hypotheses and main result}). Also, we find that to obtain our results, it suffices to assume that $G_f\in \mathcal{H}^{\varpi_0}\cap L_2(\nu_T)$ for some $\varpi_0\in(0,+\infty)$ (cf. \textbf{(H7)} in Section \ref{Hypotheses and main result}). Since
$\mathcal{H}_2\subset\mathcal{H}^\infty\subset\mathcal{H}^\varpi$, $\forall\varpi\in(0,+\infty)$ (cf. \eqref{eq H infinit H} below), it means our assumption for $G_f$ is weaker than \cite[page:2849, \textbf{(H6)}]{XZ}, where $G_f$ is assumed to be in $\mathcal{H}_2\cap L_2(\nu_T)$. To achieve these improvements, we can not follow the methods used in \cite{XZ}, but have to employ different approaches/new ideas, which surprisingly simplify the proofs than expected. We utilize the improved sufficient criteria proposed in \cite{LSZZ}, which has been proven to be an effective criteria in \cite{LSZZ, WZSIAM, ZLZ} and seems to be a more appropriate criteria compared to the one used in \cite{XZ}. We
adopt a series of technical methods, including time discretization, a cut-off argument, and relative entropy estimates of a sequence of probability measures as used in \cite{WZSIAM} to prove the convergence of the solutions to the skeleton equations. With all the discussions above, we can conclude that although both \cite{XZ} and our present paper are concerned with the general framework, our work in fact has fundamental improvements compared to \cite{XZ}. These consequently lead us to use different strategies and techniques, making the entire program a nontrivial endeavor.

The remainder of this paper is organized as follows: Section \ref{Preliminaries} provides a review of basic notations related to Poisson random measures and presents fundamental knowledge about Gelfand triples. Section \ref{Hypotheses and main result} states the hypotheses and the main result: the large deviations for \eqref{eq SPDE 01}. Section \ref{wellposedness for skeleton equation} is dedicated to proving the existence and uniqueness of solutions to the skeleton equations. Sections \ref{LDP 1} and \ref{LDP 2} are focused on proving the improved sufficient criteria.

\section{Preliminaries}\label{Preliminaries}
Set $\mathbb{N}:=\{1,2,3,\cdots\}$, $\mathbb{R}:=(-\infty,+\infty)$ and $\mathbb{R}_+:=[0,+\infty)$. For a metric space $S$, the Borel $\sigma$-field on $S$ will be written as
$\mathcal{B}(S)$. We denote by $C_c(S)$ the space of real-valued continuous functions with compact supports. Let $C([0,T];S)$ be the space of continuous functions $g:[0,T]\rightarrow S$ endowed with the uniform convergence topology. Let $D([0,T];S)$ be the space of all c\`adl\`ag functions $g:[0,T]\rightarrow S$ endowed with the Skorokhod topology.

For a locally compact Polish space $S$, the space of all Borel measures on $S$ is denoted by $M(S)$, and $M_{FC}(S)$ denotes the set of all  $\mu\in M(S)$ with $\mu(O)<+\infty$
for each compact subset $O\subseteq S$. We endow $M_{FC}(S)$ with the weakest topology such that for each $g\in C_c(S)$ the mapping
$\mu\in M_{FC}(S)\rightarrow \int_Sg(s)\mu(ds)$ is continuous. This topology is metrizable such that $M_{FC}(S)$ is a Polish space, see \cite{Budhiraja-Dupuis-Maroulas.} for more details.

\vskip 0.3cm

We fix $T>0$ throughout this paper.  Assume that $Z$ is a locally compact Polish space with a $\sigma$-finite measure $\nu\in M_{FC}(Z)$. The  probability space $(\Omega, \cF, {\mathbb F}:=\{\cF_t\}_{t\in [0,T]},\mathbb{P})$ is specified as follows.
\begin{align*}
  \Omega:=M_{FC}\big([0,T]\times Z
  \times \mathbb{R}_+\big),\qquad \cF:=\cB(\Omega).
\end{align*}
We introduce the coordinate mapping
\begin{align*}
\bar{N}\colon \Omega \rightarrow \Omega:\bar{N}(\omega)=\omega,\ \forall\omega\in\Omega.
\end{align*}
Define for each $t\in [0,T]$ the $\sigma$-algebra
\begin{align*}
\mathcal{G}_{t}:=\sigma\left\{\bar{N}((0,s]\times A):\,
0\leq s\leq t,\,A\in \mathcal{B}\big(Z\times \mathbb{R}_+\big)\right\}.
\end{align*}
For the given $\nu$, it follows from \cite[Sec.I.8]{Ikeda-Watanabe} that there exists a unique probability measure $\mathbb{P}$
 on $(\Omega,\mathcal{F})$ such that $\bar{N}$ is a Poisson random measure (PRM) on $[0,T]\times Z\times\mathbb{R}_+$ with intensity measure $\text{Leb}_T\otimes\nu\otimes \text{Leb}_\infty$, where
$\text{Leb}_T$ and $\text{Leb}_\infty$ stand for the Lebesgue measures on $[0,T]$ and $\mathbb{R}_+$ respectively.

We denote by $\mathbb{F}:=\{{\mathcal{F}}_{t}\}_{t\in[0,T]}$ the
$\mathbb{P}$-completion of $\{\mathcal{G}_{t}\}_{t\in[0,T]}$ and
$\mathcal P$ the $\mathbb{F}$-predictable $\sigma$-field
on $[0,T]\times \Omega$. The  PRM $\bar{N}$
is defined on  the (filtered) probability space $(\Omega, \cF, \mathbb{F}:=\{\cF_t\}_{t\in [0,T]},\mathbb{P})$.
The corresponding compensated PRM is denoted by $\widetilde{\bar{N}}$.

Denote
\begin{align*}
{\mathcal R_+}
:=\left\{\varphi\colon [0,T]\times \Omega\times Z\to \mathbb{R}_+: \varphi\ \text{is}
\, (\mathcal{P}\otimes\mathcal{B}(Z))/\mathcal{B}(\mathbb{R}_+)\text{-measurable}\right\}.
\end{align*}
For any $\varphi\in{\mathcal R_+}$, $N^{\varphi}:\Omega\rightarrow M_{FC}([0,T]\times Z)$ is  a
counting process  on $[0,T]\times {Z}$ defined by
   \begin{align}\label{Jump-representation}
      N^\varphi((0,t]\times A):=\int_{(0,t]\times A\times \mathbb{R}_+}1_{[0,\varphi(s,z)]}(r)\, \bar{N}(dr, dz, ds),\ 0\le t\le T,\ A\in\mathcal{B}(Z).
   \end{align}
Here $N^\varphi$ can be viewed as a controlled random measure, with $\varphi$ selecting the intensity.

Analogously, $\widetilde{N}^\varphi$ is defined by replacing $\bar{N}$ with $\widetilde{\bar{N}}$ in (\ref{Jump-representation}). When $\varphi\equiv c>0$, we write $N^\varphi=N^c$ and $\widetilde{N}^\varphi=\widetilde{N}^c$.

\vskip 0.3cm

Let $V$ be a reflexive and separable Banach space, which is densely and
continuously injected in a separable Hilbert space $(H,\ \langle\cdot,\cdot\rangle_H)$. Identifying $H$ with its dual we get
$$
V\subset H\cong H^*\subset V^*,
$$
where the star `*' denotes the dual spaces. Denote $_{V^*}\langle\cdot,\cdot\rangle_{V}$ the duality between $V^*$ and $V$, then we have
\begin{eqnarray}\label{triple}
_{V^*}\langle u,v\rangle_{V}=\langle u,v\rangle_{H},\ \ \ \forall\ u\in H,\ v\in V.
\end{eqnarray}

Now we consider the following type of SPDEs driven by a L\'evy process: for any $\epsilon\in(0,1)$,
\begin{eqnarray}\label{eq SPDE 01}
d X^\epsilon_t= \mathcal{A}(t,X^\epsilon_t)dt+\epsilon\int_{ Z}f(t,X^\epsilon_{t-},z)\widetilde{N}^{\epsilon^{-1}}(dz,dt),\ \ t\in[0,T],
\end{eqnarray}
with initial data $x\in H$,
where $\mathcal{A}:[0,T]\times V\rightarrow V^*$ is a $\mathcal{B}([0,T])\otimes\mathcal{B}(V)/\mathcal{B}(V^*)$-measurable function and $f:[0,T]\times H\times  Z\rightarrow H$ is a $\mathcal{B}([0,T])\otimes\mathcal{B}(H)\otimes\mathcal{B}( Z)/\mathcal{B}(H)$-measurable function.

\begin{definition}\label{def 02}
 An $H$-valued c\`adl\`ag $\mathbb{F}$-adapted process $X^\epsilon=\{X^\epsilon_t\}_{t\in[0,T]}$ is called a solution of \eqref{eq SPDE 01}, if for its $dt\times \mathbb{P}$-equivalent class $\widehat{X}^\epsilon$ we have
\begin{itemize}
\item[(1)] $\widehat{X}^\epsilon\in L^\alpha([0,T];V)\cap L^2([0,T];H)$, $\mathbb{P}$-a.s.;

\item[(2)] the following equality holds $\mathbb{P}$-a.s.:
\begin{eqnarray*}
X^\epsilon_t=x+\int_0^t \mathcal{A}(s,\overline{X}^\epsilon_s)ds+\epsilon\int_0^t\int_{ Z}f(s,{X}^\epsilon_{s-},z)\widetilde{N}^{\epsilon^{-1}}(dz,ds),\ \forall t\in[0,T],
\end{eqnarray*}
\end{itemize}
where $\overline{X}^\epsilon$ is any $V$-valued progressively measurable $dt\times \mathbb{P}$ version of $\widehat{X}^\epsilon$.
\end{definition}

\begin{remark}
It is  a well-known and typical conclusion in probability theory that $dt\times \mathbb{P}$-equivalent processes/versions are regarded as the same stochastic process, and it is always impossible to find one version to satisfy all required properties.
In the above definition, three $dt\times \mathbb{P}$-equivalent versions ``$X^\epsilon$, $\widehat{X}^\epsilon$, $\overline{X}^\epsilon$" are implicitly required to ensure that each version satisfies some required properties.

In the rest of this paper, when there is no danger of causing ambiguity, we denote $dt\times \mathbb{P}$-equivalent processes/versions of a given
process $X$ by itself.
\end{remark}

\section{Hypotheses and main result}\label{Hypotheses and main result}

In this paper, we assume the following conditions.
Suppose that there exist constants $\alpha>1,\ \beta\geq0,\ \theta>0,\ C>0$, a function $F\in L^1([0,T];\mathbb{R}^+)$ and a function $\rho:V\rightarrow [0,+\infty)$ which is Borel measurable and bounded on the balls,
 such that the following conditions hold for all $v,v_1,v_2\in V$ and $t\in [0,T]$:
\begin{itemize}
\item[{\bf (H1)}](Hemicontinuity) The map $s\mapsto~ _{V^*}\langle \mathcal{A}(t,v_1+s v_2),v\rangle_{V}$ is continuous on $\mathbb{R}$.

\item[{\bf (H2)}](Local monotonicity)
\begin{eqnarray*}
2~_{V^*}\langle \mathcal{A}(t,v_1)-\mathcal{A}(t,v_2),v_1-v_2\rangle_{V}
\leq
(F_t+\rho(v_2))\|v_1-v_2\|^2_H.
\end{eqnarray*}

\item[{\bf (H3)}](Coercivity)
$$
2~_{V^*}\langle \mathcal{A}(t,v),v\rangle_{V}+\theta\|v\|^\alpha_V
\leq
F_t(1 + \|v\|^2_H).
$$

\item[{\bf (H4)}](Growth)
$$
\|\mathcal{A}(t,v)\|^{\frac{\alpha}{\alpha-1}}_{V^*}
\leq
(F_t+C\|v\|^\alpha_V)(1+\|v\|_H^\beta).
$$

\item[{\bf (H5)}]
\begin{eqnarray*}
\rho(v)\leq C(1+\|v\|_V^\alpha)(1+\|v\|^\beta_H).
\end{eqnarray*}
\end{itemize}

For simplicity we write $\nu_T$ for $\text{Leb}_T\otimes\nu$.
For $p\in(0,+\infty)$, define
 \begin{eqnarray*}
\mathcal{H}_p:=&&\!\!\!\!\!\!\!\!\Big\{h:[0,T]\times{Z}\to\mathbb{R}^+:\ \exists\delta>0, s.t.\ \forall O\in\mathcal{B}([0,T])\otimes\mathcal{B}({Z})~\text{with}~\nu_T(O)<+\infty,\\
&&\qquad\qquad \mbox{we have } \ \int_O\exp(\delta h^p(t,z))\nu(dz)dt<\infty\Big\},\label{Fun h}
\end{eqnarray*}
and
\begin{eqnarray*}
L_p(\nu_T):=\Big\{h:[0,T]\times{Z}\to\mathbb{R}^+:
           \int_0^T\int_Zh^p(t,z)\nu(dz)dt<\infty\Big\}.
\end{eqnarray*}
For $\varpi\in(0,+\infty)$,
define
\begin{eqnarray*}
\mathcal{H}^\varpi:=&&\!\!\!\!\!\!\!\!\big\{h:[0,T]\times Z\rightarrow\Bbb{R}^+: \forall \Gamma\in\mathcal{B}([0,T])\otimes\mathcal{B}(Z)~\text{with}~\nu_T(\Gamma)<+\infty,~\nonumber\\
&&\ \ \ \ \ \ \ \ \ \ \text{we~have}~\int_\Gamma \exp(\varpi h(s,z))\nu(dz)ds<\infty\big\},
\end{eqnarray*}
and denote by $\mathcal{H}^\infty=\bigcap_{\varpi\in(0,+\infty)}\mathcal{H}^\varpi$. By \cite[Remark 3.2]{BCDSPA}, we have
\begin{eqnarray}\label{eq H infinit H}
\mathcal{H}_2\subset\mathcal{H}^\infty.
\end{eqnarray}

To study the large deviation principle (LDP) of \eqref{eq SPDE 01}, besides the conditions \textbf{(H1)}-\textbf{(H5)}, we further need
\begin{itemize}
  \item[{\bf (H6)}] There exists  
  $L_f \in  L_{2}(\nu_T)\cap \mathcal{H}_2$
  such that
  $$
  \|f(t,v,z)\|_H\leq L_f(t,z)(1+\|v\|_H),\ \ \forall (t,v,z)\in[0,T]\times V\times{Z}.
  $$

  \item[{\bf (H7)}]There exists $G_f\in \mathcal{H}^{\varpi_0}\cap L_2(\nu_T)$ for some $\varpi_0\in(0,+\infty)$, such that
  $$
  \|f(t,v_1,z)-f(t,v_2,z)\|_H\leq G_f(t,z)\|v_1-v_2\|_H,\ \ \forall (t,z)\in[0,T]\times{Z},\ \ v_1,v_2\in V.
  $$
\end{itemize}

\begin{remark}\label{H6 remark}
We claim that $L_{2}(\nu_T)\cap \mathcal{H}_2\subset L_{\beta+2}(\nu_T)\cap \mathcal{H}_2$.
\begin{proof}
Let $h\in\mathcal{H}_2\cap L_2(\nu_T)$. By the definition of $\mathcal{H}_2$, there exists a $\delta>0$ such that $\forall O\in\mathcal{B}([0,T])\otimes\mathcal{B}({Z})$ with $\nu_T(O)<+\infty$, we have $\int_O\exp(\delta h^2(t,z))\nu(dz)dt<+\infty$.
For this $\delta$ and  $\beta\geq0$, there exists $M>0$, which depends on $\delta$ and $\beta$, such that
\begin{eqnarray}\label{0}
  \exp(\delta y^2)\geq y^{\beta+2},\ \forall y\geq M.
\end{eqnarray}
Denote $E:=\{(t,z)\in[0,T]\times Z: h(t,z)\geq M\}$ and $E^c:=[0,T]\times Z\setminus E$. Since $h\in L_2(\nu_T)$, we have
$$\nu_T(E)\leq\frac{1}{M^2}\int_0^T\int_Z h^2(t,z)\nu(dz)dt<\infty,$$
hence by the definition of $\mathcal{H}_2$,
$$\int_E\exp(\delta h^2(t,z))\nu(dz)dt<\infty.$$
Then,
\begin{eqnarray*}
&&\!\!\!\!\!\!\!\!\int_0^T\int_Zh^{\beta+2}(t,z)\nu(dz)dt\nonumber\\
=&&\!\!\!\!\!\!\!\!\int_Eh^{\beta+2}(t,z)\nu(dz)dt+\int_{E^c}h^{\beta+2}(t,z)\nu(dz)dt\nonumber\\
\leq&&\!\!\!\!\!\!\!\!\int_E\exp(\delta h^2(t,z))\nu(dz)dt+M^\beta\int_{E^c}h^2(t,z) \nu(dz)dt\nonumber\\
\leq&&\!\!\!\!\!\!\!\!\int_E\exp(\delta h^2(t,z))\nu(dz)dt+M^\beta\int_0^T\int_Zh^2(t,z) \nu(dz)dt\nonumber\\
<&&\!\!\!\!\!\!\!\!+\infty,
\end{eqnarray*}
where we used \eqref{0} in the first inequality. Therefore, $h\in L_{\beta+2}(\nu_T)\cap \mathcal{H}_2$.
\end{proof}
\end{remark}

With a minor modification of
\cite[Theorem 1.2]{BLZnolinear} and using Remark \ref{H6 remark}, we have the following existence and uniqueness result for the solution of \eqref{eq SPDE 01}.

\begin{proposition}\label{thm solution 01}
Suppose that conditions $\textbf{(H1)}$-$\textbf{(H7)}$ hold.
Then
\eqref{eq SPDE 01} has a unique solution $X^\epsilon=\{X^\epsilon_t\}_{t\in[0,T]}$. 
%
\end{proposition}

\begin{remark}
The main differences between the proofs of Proposition \ref{thm solution 01} and \cite[Theorem 1.2]{BLZnolinear} lie in two aspects. There is no Wiener noise term in \eqref{eq SPDE 01}, which makes the proof of Proposition \ref{thm solution 01} simpler. The mapping $F$ appearing in Conditions (H2)-(H4) of this paper is a deterministic function, while in the conditions of \cite[Theorem 1.2]{BLZnolinear}, $F$ could depend on $\omega$. This leads to the following modifications:  The local monotonicity and coercivity conditions in \cite[Theorem 1.2]{BLZnolinear} are satisfied for \eqref{eq SPDE 01} with the constant $C$ replaced by an $L^1([0,T];\Bbb{R}^+)$ function, e.g., local monotonicity holds with $C$ replaced by $\int_ZG^2_f(t,z)\nu(dz)+F_t$ and coercivity holds with $C$ replaced by $F_t$.
\end{remark}


\begin{remark}
The conditions in Proposition \ref{thm solution 01} are satisfied by a very large class of SPDEs driven by a multiplicative pure jump L\'evy noise, including the stochastic porous medium equation, stochastic p-Laplace equation, stochastic Burgers type equations, stochastic 2D Navier-Stokes equations and many other stochastic hydrodynamical systems. \cite[Section 2]{BLZnolinear} presents many concrete examples to illustrate the applications of this proposition. It is omitted in this paper.
\end{remark}

In the present paper, we aim to establish a LDP for the solution of (\ref{eq SPDE 01}) as $\epsilon\rightarrow 0$ on
$D([0,T];H)$.

\vskip 0.3cm
We first state the LDP we are concerned with.
The theory of large deviations is concerned with events $A\in\mathcal{B}(D([0,T];H))$ for which probability $\mathbb{P}(X^\epsilon\in A)$
converges to
zero exponentially fast as $\epsilon\rightarrow 0$. The exponential decay rate of such probabilities is typically
expressed
in terms of a ``rate function" $I$ defined as below.
    \begin{definition}\label{Dfn-Rate function}
       \emph{\textbf{(Rate function)}} A function $I: D([0,T];H)\rightarrow[0,+\infty]$ is called a rate function on
       $D([0,T];H)$,
       if for each $M<+\infty$ the level set $\{y\in D([0,T];H):I(y)\leq M\}$ is a compact subset of $D([0,T];H)$.
       For
       $A\in \mathcal{B}(D([0,T];H))$,
       we define $I(A):=\inf_{y\in A}I(y)$.
    \end{definition}
    \begin{definition}  \label{d:LDP}
       \emph{\textbf{(Large deviation principle)}} Let $I$ be a rate function on $D([0,T];H)$. The sequence
       $\{X^\epsilon\}_{\epsilon\in(0,1)}$
       is said to satisfy a large deviation principle (LDP) on $D([0,T];H)$ with rate function $I$ if the following two
       conditions
       hold.

         a. LDP upper bound. For each closed subset $O$ of $D([0,T];H)$,
              $$
                \limsup_{\epsilon\rightarrow 0}\epsilon\log\mathbb{P}(X^\epsilon\in O)\leq-I(O).
              $$

         b. LDP lower bound. For each open subset $G$ of $D([0,T];H)$,
              $$
                \liminf_{\epsilon\rightarrow 0}\epsilon\log\mathbb{P}(X^\epsilon\in G)\geq-I(G).
              $$
    \end{definition}

Before stating our main result, we need to introduce the so-called skeleton equation. For each Borel measurable function
$g\colon [0,T]\times Z\to [0,+\infty)$, define
\begin{align*}
Q(g):=\int_{[0,T]\times Z}\ell\big(g(s,z)\big) \, \nu(dz)ds,
\end{align*}
where $\ell(x)=x\log x-x +1,\ \ell(0):=1$.
For each $N\in\Bbb{N}$, define
\begin{align*}
     S^N:=\Big\{g:[0,T]\times Z\rightarrow[0,\infty):\,Q(g)\leq N\Big\}.
\end{align*}
Any  $g\in S^N$ can be identified with a measure $\hat{g}\in M_{FC}([0,T]\times Z)$, defined by
   \begin{align*}\label{eq.corres-func-meas}
      \hat{g}(A):=\int_A g(s,z)\, \nu(dz)ds,\ \forall A\in\mathcal{B}([0,T]\times Z).
   \end{align*}
This identification induces a topology on $S^N$ under which $S^N$ is a compact space (see the Appendix of \cite{BCDSPA}), which is used throughout
the paper.

Denote
\begin{eqnarray*}\label{eq P22}
S:= \bigcup_{N\in\mathbb{N}}S^N.
\end{eqnarray*}

For any $g\in S$, consider the following deterministic PDE (called the skeleton equation): for any $t\in[0,T]$,

\begin{eqnarray}\label{eq Skeleton-1}
  Y^{g}_t=x+\int_0^t\mathcal{A}(s,Y^{g}_s)ds+\int_0^t\int_Zf(s,Y^{g}_s,z)(g(s,z)-1)\nu(dz)ds,\ \text{in} \ V^*.
\end{eqnarray}

By Proposition \ref{Prop 4.4} below, this equation has a unique solution $Y^{g}\in C([0,T];H)\cap L^\alpha([0,T];V)$.

\vskip 0.3cm

The following is the main result of this paper.
\begin{theorem}
\label{th main LDP}
 Assume that conditions $\textbf{(H1)}$-$\textbf{(H7)}$ hold.  Then the family $\{X^\epsilon\}_{\epsilon\in(0,1)}$ satisfies
a LDP on $D([0,T];H)$ with the rate function $I:D([0,T];H)\rightarrow[0,+\infty]$ defined by
     \begin{eqnarray*}\label{Rate function I}
        I(\phi):=\inf\Big\{ Q(g):~g\in S \text{ such that } Y^{g}=\phi\Big\},\ \qquad \phi\in D([0,T];H).
     \end{eqnarray*}
     Here $Y^{g}$ is the unique solution to \eqref{eq Skeleton-1}.
By convention, $\inf\emptyset=+\infty$.

\end{theorem}
\begin{proof}
Proposition \ref{Prop 4.4} allows us to define a map
\begin{eqnarray}\label{def G0}
\Gamma:S\ni g \mapsto Y^{g}\in C([0,T];H)\cap L^\alpha([0,T];V),
\end{eqnarray}
where $Y^{g}$ is the unique solution of (\ref{eq Skeleton-1}).

Let $\{Z_n\}_{n\in\mN}$ be a sequence of compact sets satisfying that $Z_n\subseteq Z$
and $ Z_n \nearrow Z$.  For each $n\in\mN$, let
\begin{eqnarray*}\label{Eq Rn}
     \mathcal{R}_{b,n}
= \Big\{\psi\in \mathcal{R}_+:
\psi(t,z,\omega)\in \begin{cases}
 [\tfrac{1}{n},n], &\text{if }z\in Z_n\\
\{1\}, &\text{if }z\in Z_n^c
\end{cases}
\text{ for all }(t,\omega)\in [0,T]\times \Omega
\Big\},
\end{eqnarray*}
and $\mathcal{R}_{b}=\bigcup _{n=1}^\infty \mathcal{R}_{b,n}$.
For any $N\in\mathbb{N}$, let $\mathcal{S}^N$ be a space of stochastic processes on $\Omega$ defined by
\begin{eqnarray*}
\mathcal{S}^N:=\{\psi\in  \mathcal{R}_{b}:\, \psi(\cdot,\cdot,\omega)\in S^N
\text{ for $\mathbb{P}$-a.e. $\omega\in \Omega$}\}.
\end{eqnarray*}

By Proposition \ref{thm solution 01}, the Yamada-Watanabe theorem and Girsonov theorem for Poisson random measures, for $\epsilon\in(0,1)$,
there exists a map
$$\Gamma^\epsilon:M_{FC}\big([0,T]\times Z\big)\rightarrow D([0,T];H)\cap L^\alpha([0,T];V)$$
 such that $X^\epsilon=\Gamma^\epsilon(N^{\epsilon^{-1}})$ is the unique solution to \eqref{eq SPDE 01}. Moreover, for any $N\in(0,\infty)$ and $\psi_\epsilon\in \mathcal{S}^N$, let
\begin{equation}\label{sol-control}
X^{\psi_\epsilon}:=\Gamma^\epsilon( N^{\epsilon^{-1}\psi_\epsilon}),
\end{equation}
then $X^{\psi_\epsilon}$ is the unique solution of the integral equation (or generally called ``controlled SPDE"):

\begin{eqnarray}\label{sol-control 01}
X^{\psi_\epsilon}_t=&&\!\!\!\!\!\!\!\!x+ \int_0^t\mathcal{A}(s,X^{\psi_\epsilon}_s)ds+\epsilon\int_0^t\int_{Z}f(s,X^{\psi_\epsilon}_{s-},z)\Big(N^{\epsilon^{-1}\psi_\epsilon}(dz,ds)-\epsilon^{-1}\nu(dz)ds\Big)\nonumber\\
=&&\!\!\!\!\!\!\!\!x+ \int_0^t\mathcal{A}(s,X^{\psi_\epsilon}_s)ds
  +
  \epsilon\int_0^t\int_{ Z}f(s,X^{\psi_\epsilon}_{s-},z)\widetilde{N}^{\epsilon^{-1}\psi_\epsilon}(dz,ds)\nonumber\\
  &&\!\!\!\!\!\!\!\!+\int_0^t\int_{ Z}f(s,X^{\psi_\epsilon}_{s},z)\big(\psi_\epsilon(s,z)-1\big)\nu(dz)ds
  ,\ \forall t\in[0,T].
\end{eqnarray}

For the details of the proof of the above result, we refer \cite[Lemma 7.1]{BPZJEMS}.

According to \cite[Theorem 4.4]{LSZZ}, which is an adaption of the original results given in \cite[Theoerm 4.2]{Budhiraja-Dupuis-Maroulas.} and \cite[Theorems 2.3 and 2.4]{BCDSPA}, to complete the proof of the theorem,
it is sufficient to verify the following two claims:
\begin{itemize}
  \item[(\textbf{LDP1})] For any given $N\in\mathbb{N}$, let $\psi,\psi_n \in S^N$, $n\in\mathbb{N}$, be such that
$\psi_n\rightarrow \psi$ in $S^N$ as $n\rightarrow+\infty$. Then
$$
\lim_{n\rightarrow\infty}\sup_{t\in[0,T]}\|{\Gamma}(\psi_n)(t)-{\Gamma}(\psi)(t)\|_{H}=0.
$$
  \item[(\textbf{LDP2})] For any given $N\in\mathbb{N}$, let $\{\psi_\epsilon,~\epsilon>0\}\subset \mathcal{S}^N$. Then, for any $\delta>0$,
  \begin{align*}
\lim_{\epsilon\rightarrow0}\mathbb{P}\left(\sup_{t\in[0,T]}\|\Gamma^\epsilon( N^{\epsilon^{-1}\psi_\epsilon})(t)-{\Gamma}(\psi_{\epsilon})(t))\|_{H}\geq \delta\right)=0.
\end{align*}
\end{itemize}

The verification of (\textbf{LDP1}) will be given in Proposition \ref{Prop LDP1}; see Section \ref{LDP 1}. (\textbf{LDP2}) will
be established in Proposition \ref{Prop LDP2}; see Section \ref{LDP 2}.
\end{proof}

Finally, we introduce the following results, which will be used later.

\begin{lemma}\label{lemma-est-ass}
\begin{enumerate}
       \item[(i)] Let $\chi\in\mathcal{H}_2\cap L_2(\nu_T)$, we have
    \begin{eqnarray}\label{eq lemma-est-ass 01}
    \sup _{\hbar \in S^{N}}\int_0^T\int_{Z}|\chi(t,z)|^{2}(\hbar(t, z)+1) \nu(dz) d t<\infty,~~\forall N\in\Bbb{N}.
    \end{eqnarray}
  \item[(ii)] Let $\chi\in\mathcal{H}^{\varpi}\cap L_2(\nu_T)$ for some $\varpi\in(0,+\infty)$, we have
    \begin{eqnarray}\label{eq lemma-est-ass 02}
   \sup _{\hbar \in S^{N}}\int_0^T \int_{Z}\chi(t,z)|\hbar (t, z)-1| \nu(dz) dt<\infty,~~\forall N\in\Bbb{N}.
    \end{eqnarray}
    \item[(iii)]Let $\chi\in\mathcal{H}^{\infty}\cap L_2(\nu_T)$ and $N\in\Bbb{N}$, we have
  \begin{eqnarray}\label{eq lemma-est-ass 03}
   \lim_{\delta\rightarrow0}\sup _{\hbar \in S^{N}}\sup_{O\in\mathcal{B}([0,T])\atop{\rm Leb}_T(O)\leq \delta}\int_O \int_{Z}\chi(t,z)|\hbar (t, z)-1| \nu(dz) dt=0.
    \end{eqnarray}
   \item[(iv)]Let $\chi\in\mathcal{H}^{\infty}\cap L_2(\nu_T)$ and $N\in\Bbb{N}$, then for any $\varepsilon>0$ there exists a compact set $K_\varepsilon\subset Z$ such that
\begin{eqnarray}\label{eq zhai 1}
\sup_{\hbar\in S^N}\int_0^T\int_{K_\varepsilon^c}\chi(t,z)|\hbar(t,z)-1|\nu(dz)dt\leq\varepsilon,
\end{eqnarray}
where $K_\varepsilon^c$ is the complement of $K_\varepsilon$.
 \item[(v)] Let $\chi\in\mathcal{H}^\infty$, $K$ be a compact subset of $Z$ and $N\in\Bbb{N}$, we have
  \begin{eqnarray}\label{Lf}
   \lim_{j\rightarrow\infty}\sup_{\hbar\in S^N}\int_0^T\int_K \chi(t,z)\cdot 1_{\{\chi(t,z)\geq j\}}(t,z)\hbar(t,z)\nu(dz)dt=0.
    \end{eqnarray}
\end{enumerate}
\end{lemma}

\begin{remark}
The proofs of \eqref{eq lemma-est-ass 01},  \eqref{eq lemma-est-ass 02}, \eqref{eq zhai 1} and \eqref{Lf} can be found in the proofs of \cite[(3.3), (3.4), (3.23), (3.26)]{BCDSPA} respectively and use \eqref{eq H infinit H} and the fact that for any compact subset $K\subset Z$, $\nu(K)<+\infty$. \eqref{eq lemma-est-ass 03} is quoted from \cite[Remark 2]{YZZ}, 
 which is a little more general than \cite[(3.5)]{BCDSPA}, i.e.,
\begin{eqnarray*}\label{Ineq SN}
\lim _{\delta \rightarrow 0} \sup _{\hbar \in S^{N}} \sup _{0\leq s\leq t\leq T\atop t-s \leq \delta} \int_s^t\int_Z \chi(t,z)|\hbar(r,z)-1| \nu(dz)  dr=0.
\end{eqnarray*}
\end{remark}

\section{Well-posedness for the skeleton equation}\label{wellposedness for skeleton equation}
Before verifying the \textbf{(LDP1)} and \textbf{(LDP2)}, we need to prove the following proposition, i.e., the existence and uniqueness of solutions to the skeleton equation \eqref{eq Skeleton-1}.
\begin{proposition}\label{Prop 4.4}
Suppose that conditions $\textbf{(H1)}$-$\textbf{(H7)}$ hold. For any $g\in S$, there exists a unique solution $Y^{g}\in C([0,T];H)\cap L^\alpha([0,T];V)$
satisfying \eqref{eq Skeleton-1}. Moreover,
\begin{eqnarray}\label{eq uniform Skelekon}
\sup_{g\in S^N}\Big(\sup_{t\in[0,T]}\|Y^{g}_t\|^2_H+\int_0^T\|Y^{g}_t\|_V^\alpha dt\Big)\leq C_N<+\infty,~~\forall N\in\mathbb{N},
\end{eqnarray}
where $C_N$ is defined as in \eqref{uniform Xg}.
\end{proposition}

\begin{proof}
The proof is mainly divided into three steps.

(\emph{Step 1})~~For $J\in C([0,T];H)\cap L^\alpha([0,T];V)$, consider the following equation:
\begin{equation}\label{ZY}
\left\{ \begin{aligned}
&dZ^J_t=\mathcal{A}(t,Z^J_t)dt+\int_Zf(t,J(t),z)(g(t,z)-1)\nu(dz)dt,\  t\in [0,T],\\
&Z^J_0=x\in H.
\end{aligned} \right.
\end{equation}
In this step, we aim to prove the existence and uniqueness of solutions to \eqref{ZY} with $Z^J\in C([0,T];H)\cap L^\alpha([0,T];V)$ satisfying
\begin{eqnarray}\label{ZY integral}
Z^J_t=x+\int_0^t\mathcal{A}(s,Z^J_s)ds+\int_0^t\int_Zf(s,J(s),z)(g(s,z)-1)\nu(dz)ds~~\text{in}~V^*,
\end{eqnarray}
and
\begin{eqnarray}\label{uniform ZY}
\sup_{t\in [0,T]}\|Z^J_t\|_H^2+\theta\int_0^T\|Z^J_s\|_V^\alpha ds\leq C_T,
\end{eqnarray}
where $C_T$ is defined as in \eqref{Cl}.

Assume that $\{e_1, e_2, ...\}\subset V$ is an orthonormal basis of $H$ such that $span\{e_1, e_2, ...\}$ is dense in $V$. Denote $H_n:=span\{e_1,...,e_n\}$. Let $P_n: V^*\rightarrow H_n$ be defined by
$$P_ny:=\sum_{i=1}^n~ _{V^*}\langle y,e_i\rangle_V e_i,~~y\in V^*.$$
It is easy to see that $P_n|_H$ is just the orthogonal projection of $H$ onto $H_n$, and, for $u_1,u_2\in V$, $v\in H_n$, we have
\begin{eqnarray}\label{eq 20230828 01}
&&\!\!\!\!\!\!\!\!~_{V^*}\langle P_n \mathcal{A}(t,u_1)+\int_ZP_nf(t,u_2,z)(g(t,z)-1)\nu(dz),v\rangle_V\nonumber\\
=&&\!\!\!\!\!\!\!\!\langle P_n \mathcal{A}(t,u_1)+\int_ZP_nf(t,u_2,z)(g(t,z)-1)\nu(dz),v\rangle_H\\
=&&\!\!\!\!\!\!\!\!~_{V^*}\langle \mathcal{A}(t,u_1)+\int_Zf(t,u_2,z)(g(t,z)-1)\nu(dz),v\rangle_V.\nonumber
\end{eqnarray}
For each $n\in\Bbb{N}$, consider the following equation on $H_n$:
\begin{equation*}
\left\{ \begin{aligned}
&dZ^{J,n}_t=P_n\mathcal{A}(t,Z^{J,n}_t)dt+\int_ZP_nf(t,J(t),z)(g(t,z)-1)\nu(dz)dt,\\
&Z^{J,n}_0=P_nx\in H.
\end{aligned} \right.
\end{equation*}
From \cite[Theorem 3.1.1]{LR}, we know that there exists a unique solution $Z^{J,n}=\{Z^{J,n}_t\}_{t\in[0,T]}$ to the above equation satisfying the following integral equation:
\begin{eqnarray*}
Z^{J,n}_t=P_nx+\int_0^t P_n\mathcal{A}(s,Z^{J,n}_s)ds+\int_0^t\int_ZP_nf(s,J(s),z)(g(s,z)-1)\nu(dz)ds,~~\forall t\in[0,T].
\end{eqnarray*}
Applying the chain rule to $\|Z^{J,n}_t\|_H^2$, by \textbf{(H3)}, \textbf{(H6)} and \eqref{eq 20230828 01}, we have
\begin{eqnarray*}\label{ZYn}
\|Z^{J,n}_t\|_H^2=&&\!\!\!\!\!\!\!\!\|P_nx\|_H^2+2\int_0^t~_{V^*}\langle P_n\mathcal{A}(s,Z^{J,n}_s),Z^{J,n}_s\rangle_V ds\nonumber\\
&&\!\!\!\!\!\!\!\!+2\int_0^t~_{V^*}\langle \int_Z P_nf(s,J(s),z)\big(g(s,z)-1\big)\nu(dz),Z^{J,n}_s\rangle_V ds\nonumber\\
\leq&&\!\!\!\!\!\!\!\!\|x\|_H^2+\int_0^t F_sds+\int_0^tF_s\cdot\|Z^{J,n}_s\|_H^2ds-\int_0^t\theta\|Z^{J,n}_s\|_V^\alpha ds\nonumber\\
&&\!\!\!\!\!\!\!\!+2\int_0^t\int_ZL_f(s,z)(1+\|J(s)\|_H)\cdot\|Z^{J,n}_s\|_H\cdot |g(s,z)-1|\nu(dz)ds\nonumber\\
\leq&&\!\!\!\!\!\!\!\!\|x\|_H^2+\int_0^t F_sds+\int_0^tF_s\cdot\|Z^{J,n}_s\|_H^2ds-\int_0^t\theta\|Z^{J,n}_s\|_V^\alpha ds\nonumber\\
&&\!\!\!\!\!\!\!\!+2\big(1+\sup_{s\in[0,T]}\|J(s)\|_H\big)\int_0^tL_f(s)\cdot\big(\|Z^{J,n}_s\|^2_H+1)ds,~~\forall t\in[0,T],
\end{eqnarray*}
where
$$L_f(s):=\int_ZL_f(s,z)|g(s,z)-1|\nu(dz).$$
By \eqref{eq lemma-est-ass 02} we know that $L_f\in L^1([0,T];\Bbb{R}^+)$. Since $F\in L^1([0,T];\Bbb{R}^+)$ and $J\in C([0,T];H)$, by Gronwall's inequality, we have \begin{eqnarray*}
\sup_{t\in [0,T]}\|Z^{J,n}_t\|_H^2+\theta\int_0^T\|Z^{J,n}_s\|_V^\alpha ds\leq C_T,
\end{eqnarray*}
where
\begin{eqnarray}\label{Cl}
C_l=&&\!\!\!\!\!\!\!\!\Big(\|x\|_H^2+\int_0^l F_sds+2\big(1+\sup_{s\in [0,l]}\|J(s)\|_H\big)\int_0^lL_f(s)ds\Big)\nonumber\\
&&\!\!\!\!\!\!\!\!\cdot e^{\int_0^l F_sds+2(1+\sup_{s\in [0,l]}\|J(s)\|_H)\int_0^lL_f(s)ds},~~~~~l\in[0,T].
\end{eqnarray}
Following the similar ideas as \cite[Theorem 4.1]{BLZnolinear}, we obtain the existence and uniqueness of solutions to \eqref{ZY} satisfying \eqref{ZY integral} and \eqref{uniform ZY}. Besides, \eqref{uniform ZY} still holds with $T$ replaced by any $l\in [0,T]$.

(\emph{Step 2})~~Choosing $M\geq 4(1+\|x\|_H^2)$, by \eqref{Cl} there exists $0<l_0<T$ such that
\begin{eqnarray*}
\Big(\|x\|_H^2+\int_0^{l_0} F_sds+2(1+M)\int_0^{l_0}L_f(s)ds\Big)\cdot e^{\int_0^{l_0} F_sds+2(1+M)\int_0^{l_0}L_f(s)ds}\leq M.
\end{eqnarray*}
Define
$$\Lambda_{l_0,M}:=\{\chi\in C([0,l_0];H)\cap L^\alpha([0,l_0];V):\sup_{s\in[0,l_0]}\|\chi(s)\|_H+\theta\int_0^{l_0}\|\chi(s)\|_V^\alpha ds\leq M\}.$$
From \eqref{uniform ZY} we see that for any $J\in \Lambda_{l_0,M}$, one has $Z^J\in\Lambda_{l_0,M}$. Define a function $d:\Lambda_{l_0,M}\times \Lambda_{l_0,M}\rightarrow [0,+\infty)$ by
$$d(\chi^1,\chi^2)=\sup_{s\in[0,l_0]}\|\chi^1(s)-\chi^2(s)\|_H,~~\forall \chi^1,\chi^2\in \Lambda_{l_0,M}.$$
It is easy to see that $d$ is a metric on $\Lambda_{l_0,M}$. In fact, $(\Lambda_{l_0,M}, d)$ is a complete metric space. Before clarifying this, we remark that in this paper $\Lambda_{l_0,M}$ is not equipped with the following usual metric $$\tilde{d}(\chi^1,\chi^2):=\sup_{s\in[0,l_0]}\|\chi^1(s)-\chi^2(s)\|_H+\Big(\int_0^{l_0}\|\chi^1(s)-\chi^2(s)\|_V^\alpha ds\Big)^{\frac{1}{\alpha}},~~\forall \chi^1,\chi^2\in \Lambda_{l_0,M}.$$
Because as seen later in this step (see the arguments under \eqref{Z1Z2}), we only prove that the unique solution to \eqref{ZY} on $[0,l_1]$ (cf. \eqref{l1} for $l_1$) is a contraction from $(\Lambda_{l_1,M},d)$ to $(\Lambda_{l_1,M},d)$, which is sufficient to use the Banach fixed-point theorem. To clarify that $(\Lambda_{l_0,M}, d)$ is complete, we need to prove that every Cauchy sequence in $(\Lambda_{l_0,M},d)$ converges and its limit is in $\Lambda_{l_0,M}$. Suppose that $\{\chi_n\}_{n\in\Bbb{N}}$ is a Cauchy sequence  in $(\Lambda_{l_0,M},d)$, hence also a Cauchy sequence in $(C([0,l_0];H),\sup_{s\in[0,l_0]}\|\cdot\|_H)$. Then
there exists some $\chi\in C([0,l_0];H)$ such that as $n\rightarrow+\infty$,
$$\chi_n\longrightarrow \chi,~~\text{strongly in}~C([0,l_0];H)\subset L^\alpha([0,l_0];H).$$
Since $\theta\int_0^{l_0}\|\chi_n(s)\|_V^\alpha ds\leq M$, $\forall n\in\Bbb{N}$, there exists some $\bar{\chi}\in L^\alpha([0,l_0];V)$ such that as $n\rightarrow+\infty$,
$$\chi_n\longrightarrow\bar{\chi},~~\text{weakly in}~ L^\alpha([0,l_0];V)\subset L^\alpha([0,l_0];H).$$
Therefore, $\chi=\bar{\chi}$, $ds$-a.s.. Since $$\sup_{s\in[0,l_0]}\|\chi_n(s)\|_H+\theta\int_0^{l_0}\|\chi_n(s)\|_V^\alpha ds\leq M,$$
by weak lower semicontinuity of the norms, we may pass to the limit and get
$$
\sup_{s\in[0,l_0]}\|\chi(s)\|_H+\theta\int_0^{l_0}\|\chi(s)\|_V^\alpha ds\leq M.
$$

From Step 1, we know that for any $J_1, J_2\in \Lambda_{l_0,M}$, there exist  unique solutions to \eqref{ZY} with $J$ replaced by $J_1$ and $J_2$,  for simplicity, denoted as $Z^1:=Z^{J_1}$ and $Z^2:=Z^{J_2}$, respectively. We have for $t\in[0,l_0]$,
\begin{eqnarray*}
&&\!\!\!\!\!\!\!\!Z^1_t-Z^2_t\nonumber\\
=&&\!\!\!\!\!\!\!\!\int_0^t\mathcal{A}(s,Z^1_s)-\mathcal{A}(s,Z^2_s)ds+\int_0^t\int_Z\big(f(s,J_1(s),z)-f(s,J_2(s),z)\big)(g(s,z)-1)\nu(dz)ds.
\end{eqnarray*}
Applying the chain rule to $\|Z^1_t-Z^2_t\|_H^2$, by \textbf{(H2)} and \textbf{(H7)},
\begin{eqnarray}\label{Z1Z2}
&&\!\!\!\!\!\!\!\! \|Z^1_t-Z^2_t\|_H^2\nonumber\\
=&&\!\!\!\!\!\!\!\! 2\int_0^t ~_{V^*}\langle \mathcal{A}(s,Z^1_s)-\mathcal{A}(s,Z^2_s), Z^1_s-Z^2_s\rangle_Vds\nonumber\\
&&\!\!\!\!\!\!\!\!+2\int_0^t\Big\langle\int_Z\big(f(s,J_1(s),z)-f(s,J_2(s),z)\big)(g(s,z)-1)\nu(dz),Z^1_s-Z^2_s\Big\rangle_Hds\nonumber\\
\leq&&\!\!\!\!\!\!\!\!\int_0^t(F_s+\rho(Z^2_s))\|Z^1_s-Z^2_s\|_H^2ds\nonumber\\
&&\!\!\!\!\!\!\!\!+2\int_0^t\int_ZG_f(s,z)\|J_1(s)-J_2(s)\|_H\cdot|g(s,z)-1|\cdot\|Z^1_s-Z^2_s\|_H\nu(dz)ds\nonumber\\
\leq&&\!\!\!\!\!\!\!\!\int_0^t\big(F_s+\rho(Z^2_s)\big)\|Z^1_s-Z^2_s\|_H^2ds\nonumber\\
&&\!\!\!\!\!\!\!\!+\int_0^t\int_ZG_f(s,z)\|J_1(s)-J_2(s)\|_H^2\cdot|g(s,z)-1|\nu(dz)ds\nonumber\\
&&\!\!\!\!\!\!\!\!+\int_0^t\int_ZG_f(s,z)\|Z^1_s-Z^2_s\|_H^2\cdot|g(s,z)-1|\nu(dz)ds,~~\forall t\in[0,l_0].
\end{eqnarray}

By Gronwall's inequality, we know that for $l\in[0,l_0]$,
\begin{eqnarray*}
&&\!\!\!\!\!\!\!\! \sup_{t\in[0,l]}\|Z^1_t-Z^2_t\|_H^2\nonumber\\
\leq&&\!\!\!\!\!\!\!\!\sup_{t\in[0,l]}\|J_1(t)-J_2(t)\|_H^2\cdot\int_0^lG_f(s)ds\cdot \exp\Big\{\int_0^lF_s+\rho(Z^2_s)+G_f(s)ds\Big\},
\end{eqnarray*}
where
$$G_f(s):=\int_ZG_f(s,z)|g(s,z)-1|\nu(dz),$$
which by \eqref{eq lemma-est-ass 02} one has $G_f\in L^1([0,T];\Bbb{R}^+)$. Since $F\in L^1([0,T];\Bbb{R}^+)$, by \textbf{(H5)} and \eqref{uniform ZY},
\begin{eqnarray*}\label{exp}
&&\!\!\!\!\!\!\!\!\exp\Big\{\int_0^{l_0}F_s+\rho(Z^2_s)+G_f(s)ds\Big\}\nonumber\\
\leq&&\!\!\!\!\!\!\!\!\exp\Big\{\int_0^{l_0}F_s+C(1+\|Z^2_s\|_V^\alpha)(1+\|Z^2_s\|_H^\beta)+G_f(s)ds\Big\}\nonumber\\
:=&&\!\!\!\!\!\!\!\!C_{M,l_0,\theta,\int_0^TF_sds,\int_0^TG_f(s)ds}<+\infty.
\end{eqnarray*}
Therefore, choosing $l_1\in(0,l_0]$ such that
\begin{eqnarray}\label{l1}
\int_0^{l_1}G_f(s)ds\cdot C_{M,l_0,\theta,\int_0^TF_sds,\int_0^TG_f(s)ds}\leq\frac{1}{2},
\end{eqnarray}
which means the unique solution to \eqref{ZY} on $[0,l_1]$ can be regarded as a map from $\Lambda_{l_1,M}$ to $\Lambda_{l_1,M}$ and it is a contraction. Hence,  by the Banach fixed-point theorem, there exists a unique solution to \eqref{eq Skeleton-1}  on $[0,l_1]$ in the space $C([0,l_1];H)\cap L^\alpha([0,l_1];V)$.

(\emph{Step 3})~~Assume that there exists a unique solution to \eqref{eq Skeleton-1} on $[0,T]$, denote it as $Y^g$. Applying the chain rule to $\|Y^g_t\|_H^2$, by \eqref{triple}, \textbf{(H3)} and \textbf{(H6)},
\begin{eqnarray*}
\|Y^g_t\|_H^2=&&\!\!\!\!\!\!\!\!\|x\|_H^2+2\int_0^t~_{V^*}\langle \mathcal{A}(s,Y^g_s),Y^g_s\rangle_V ds\nonumber\\
&&\!\!\!\!\!\!\!\!+2\int_0^t~_{V^*}\langle \int_Z f(s,Y^g_s,z)\big(g(s,z)-1\big)\nu(dz),Y^g_s\rangle_V ds\nonumber\\
\leq&&\!\!\!\!\!\!\!\!\|x\|_H^2+\int_0^t F_sds+\int_0^tF_s\cdot\|Y^g_s\|_H^2ds-\int_0^t\theta\|Y^g_s\|_V^\alpha ds\nonumber\\
&&\!\!\!\!\!\!\!\!+4\int_0^tL_f(s)ds+4\int_0^t L_f(s)\cdot\|Y^g_s\|^2_Hds,~~~\forall t\in[0,T].
\end{eqnarray*}
Then,
\begin{eqnarray*}
&&\!\!\!\!\!\!\!\!\sup_{t\in[0,T]}\|Y^g_t\|_H^2+\theta\int_0^T\|Y^g_s\|_V^\alpha ds\nonumber\\
\leq&&\!\!\!\!\!\!\!\!\|x\|_H^2+\int_0^T F_s+4L_f(s)ds+\int_0^T(F_s+4L_f(s))\cdot\|Y^g_s\|_H^2ds.
\end{eqnarray*}
By Gronwall's inequality and \eqref{eq lemma-est-ass 02},
\begin{eqnarray}\label{uniform Xg}
&&\!\!\!\!\!\!\!\!\sup_{t\in[0,T]}\|Y^g_t\|_H^2+\int_0^T\|Y^g_s\|_V^\alpha ds\nonumber\\
\leq&&\!\!\!\!\!\!\!\!\big(\|x\|_H^2+\int_0^TF_s+4L_f(s)ds\big)\cdot e^{\int_0^TF_s+4L_f(s)ds}\nonumber\\
:=&&\!\!\!\!\!\!\!\!C_N,
\end{eqnarray}
 where $C_N$ is dependent on $\int_0^TF_sds$, $\int_0^TL_f(s)ds$, $N$ and $\|x\|_H$, but independent of $g$.

Combing the results in Step 2 and Step 3, using the standard arguments, we know that there exists a unique solution $Y^g\in C([0,T];H)\cap L^\alpha([0,T];V)$ satisfying \eqref{eq Skeleton-1} on $[0,T]$. \eqref{eq uniform Skelekon} follows from \eqref{uniform Xg}.
\end{proof}

\section{The verification of \textbf{(LDP1)}}\label{LDP 1}
Recall from \eqref{def G0} that
\begin{eqnarray*}
\Gamma:S\ni g \mapsto Y^{g}\in C([0,T];H)\cap L^\alpha([0,T];V),
\end{eqnarray*}
where $Y^{g}$ is the unique solution of the skeleton equation \eqref{eq Skeleton-1}. In this section, we aim to prove the following proposition.
\bp\label{Prop LDP1}
For any given $N\in\mathbb{N}$, let $\psi,\psi_n \in S^N$, $n\in\mathbb{N}$, be such that
$\psi_n\rightarrow \psi$ in $S^N$ as $n\rightarrow+\infty$. Then
$$
\lim_{n\rightarrow\infty}\sup_{t\in[0,T]}\|{\Gamma}(\psi_n)(t)-{\Gamma}(\psi)(t)\|_{H}=0.
$$

\begin{proof}
From Proposition \ref{Prop 4.4}, we know that $X^\psi:=\Gamma(\psi)$ is the unique solution to the following deterministic PDE
\begin{equation*}
\left\{ \begin{aligned}
&dX^\psi_t=\mathcal{A}(t,X^\psi_t)dt+\int_Zf(t,X^\psi_t,z)(\psi(t,z)-1)\nu(dz)dt,\\
&X^\psi_0=x\in H,
\end{aligned} \right.
\end{equation*}
while $X^{\psi_n}:=\Gamma(\psi_n)$ is the unique solution to the following deterministic PDE
\begin{equation*}
\left\{ \begin{aligned}
&dX^{\psi_n}_t=\mathcal{A}(t,X^{\psi_n}_t)dt+\int_Zf(t,X^{\psi_n}_t,z)(\psi_n(t,z)-1)\nu(dz)dt,\\
&X^{\psi_n}_0=x\in H.
\end{aligned} \right.
\end{equation*}
Applying the chain rule to $\|X^{\psi_n}_t-X^\psi_t\|_H^2$, by \eqref{triple},
\begin{eqnarray}\label{chain}
&&\!\!\!\!\!\!\!\!\|X^{\psi_n}_t-X^\psi_t\|_H^2\nonumber\\
=&&\!\!\!\!\!\!\!\!2\int_0^t~_{V^*}\langle\mathcal{A}(s,X^{\psi_n}_s)-\mathcal{A}(s,X^{\psi}_s),X^{\psi_n}_s-X^\psi_s\rangle_Vds\nonumber\\
&&\!\!\!\!\!\!\!\!+2\int_0^t\big\langle\int_Zf(s,X^{\psi_n}_s,z)\big(\psi_n(s,z)-1\big)-f(s,X^\psi_s,z)\big(\psi(s,z)-1\big)\nu(dz),\nonumber\\
&&\ \ \ \ \ \ \ \ \ \ \ \ \ \ \ \ \ \ \ \ \ \ \ \ \ \ \ \ \ \ \ \ \ \ \ \ \ \ \ \ \ \ \ \ \ \ \ \ \ X^{\psi_n}_s-X^\psi_s\big\rangle_Hds,~~~\forall t\in[0,T].
\end{eqnarray}
By \textbf{(H2)}, \textbf{(H5)} and \eqref{eq uniform Skelekon},
\begin{eqnarray}\label{chain1}
&&\!\!\!\!\!\!\!\!2\int_0^t~_{V^*}\langle\mathcal{A}(s,X^{\psi_n}_s)-\mathcal{A}(s,X^{\psi}_s),X^{\psi_n}_s-X^\psi_s\rangle_Vds\nonumber\\
\leq&&\!\!\!\!\!\!\!\!2\int_0^t\big(F_s+C(1+\|X^\psi_s\|_V^\alpha)(1+\|X^\psi_s\|_H^\beta)\big)\cdot\|X^{\psi_n}_s-X^\psi_s\|_H^2ds\nonumber\\
=&&\!\!\!\!\!\!\!\!2\int_0^t(F_s+C)\cdot\|X^{\psi_n}_s-X^\psi_s\|_H^2ds+2C\int_0^t\|X^\psi_s\|_V^\alpha\cdot\|X^{\psi_n}_s-X^\psi_s\|_H^2ds\nonumber\\
&&\!\!\!\!\!\!\!\!+2C\int_0^t\|X^\psi_s\|_H^\beta\cdot\|X^{\psi_n}_s-X^\psi_s\|_H^2ds\nonumber\\
&&\!\!\!\!\!\!\!\!+2C\int_0^t\|X^\psi_s\|_V^\alpha\cdot\|X^\psi_s\|_H^\beta\cdot\|X^{\psi_n}_s-X^\psi_s\|_H^2ds\nonumber\\
\leq&&\!\!\!\!\!\!\!\!2\int_0^t(F_s+C)\cdot\|X^{\psi_n}_s-X^\psi_s\|_H^2ds+2C\int_0^t\|X^\psi_s\|_V^\alpha\cdot\|X^{\psi_n}_s-X^\psi_s\|_H^2ds\nonumber\\
&&\!\!\!\!\!\!\!\!+2C\sup_{\hbar\in S^N}\sup_{t\in[0,T]}\|X^\hbar_t\|_H^\beta\cdot\int_0^t\|X^{\psi_n}_s-X^\psi_s\|_H^2ds\nonumber\\
&&\!\!\!\!\!\!\!\!+2C\sup_{\hbar\in S^N}\sup_{t\in[0,T]}\|X^\hbar_t\|_H^\beta\cdot\int_0^t\|X^\psi_s\|_V^\alpha\cdot\|X^{\psi_n}_s-X^\psi_s\|_H^2ds\nonumber\\
:=&&\!\!\!\!\!\!\!\!\int_0^t2(F_s+C_{N,\beta}\|X^\psi_s\|_V^\alpha+C_{N,\beta})\cdot\|X^{\psi_n}_s-X^\psi_s\|_H^2ds,\ \ \forall t\in[0,T].
\end{eqnarray}
Here $C_{N,\beta}$ only depends on $C_N$ appearing in \eqref{eq uniform Skelekon} and  $\beta$.
For simplicity, denote the second term in the right hand-side of \eqref{chain} as $I_n(t)$ and rewrite it as following:
\begin{eqnarray}\label{In}
&&\!\!\!\!\!\!\!\!I_n(t)\nonumber\\
:=&&\!\!\!\!\!\!\!\!2\int_0^t\Big\langle\int_Zf(s,X^\psi_s,z)\big((\psi_n(s,z)-1)-(\psi(s,z)-1)\big)\nu(dz),X^{\psi_n}_s-X^{\psi}_s\Big\rangle_Hds\nonumber\\
&&\!\!\!\!\!\!\!\!+2\int_0^t\Big\langle\int_Z\big(f(s,X^{\psi_n}_s,z)-f(s,X^{\psi}_s,z)\big)(\psi_n(s,z)-1)\nu(dz),X^{\psi_n}_s-X^{\psi}_s\Big\rangle_Hds\nonumber\\
:=&&\!\!\!\!\!\!\!\!Q_{n,1}(t)+Q_{n,2}(t).
\end{eqnarray}
Denote
$$G_n(s):=\int_ZG_f(s,z)|\psi_n(s,z)-1|\nu(dz).$$
From \eqref{eq lemma-est-ass 02} we know that
\begin{eqnarray}\label{eq S 20230915} \sup_{n\in\mathbb{N}}\int_0^TG_n(s)ds<\infty.
\end{eqnarray}
For $Q_{n,2}(t)$, from \textbf{(H7)} we have
\begin{eqnarray}\label{Qn2}
|Q_{n,2}(t)|\leq 2\int_0^t G_n(s)\|X^{\psi_n}_s-X^{\psi}_s\|_H^2ds.
\end{eqnarray}
Substituting \eqref{chain1}-\eqref{Qn2} to \eqref{chain}, we get
\begin{eqnarray*}
&&\!\!\!\!\!\!\!\!\|X^{\psi_n}_t-X^{\psi}_t\|_H^2\nonumber\\
\leq&&\!\!\!\!\!\!\!\!\int_0^t2\big(F_s+C_{N,\beta}\|X^\psi_s\|_V^\alpha+C_{N,\beta}+G_n(s)\big)\cdot\|X^{\psi_n}_s-X^{\psi}_s\|_H^2ds+Q_{n,1}(t),~~~\forall t\in[0,T].
\end{eqnarray*}
From \eqref{eq uniform Skelekon}
we know that $\|X^\psi\|_V^\alpha\in L^1([0,T];\Bbb{R}^+)$, hence by Gronwall's inequality and \eqref{eq S 20230915}, we obtain
\begin{eqnarray}\label{QN1}
&&\!\!\!\!\!\!\!\!\sup_{t\in[0,T]}\|X^{\psi_n}_t-X^{\psi}_t\|_H^2\nonumber\\
\leq&&\!\!\!\!\!\!\!\!\sup_{t\in[0,T]}|Q_{n,1}(t)|\cdot e^{2\int_0^TF_s+C_{N,\beta}\|X^\psi_s\|_V^\alpha+C_{N,\beta}+G_n(s)ds}\nonumber\\
:=&&\!\!\!\!\!\!\!\!C_{G_f,N,\beta,\alpha,T}\cdot\sup_{t\in[0,T]}|Q_{n,1}(t)|.
\end{eqnarray}
Here $C_{G_f,N,\beta,\alpha,T}$ is independent of $n$.

Now let us estimate $|Q_{n,1}(t)|$. By Lemma \ref{lemma-est-ass}, we know that for any $\varepsilon>0$, there exists a compact set $K_\varepsilon\subset Z$ such that \eqref{eq zhai 1} holds.
We rewrite
\begin{eqnarray}\label{Qn1}
&&\!\!\!\!\!\!\!\!Q_{n,1}(t)\nonumber\\
=&&\!\!\!\!\!\!\!\!2\int_0^t\int_{K_\varepsilon}\Big\langle f(s,X^\psi_s,z)\big(\psi_n(s,z)-\psi(s,z)\big),X^{\psi_n}_s-X^{\psi}_s\Big\rangle_H\nu(dz)ds\nonumber\\
&&\!\!\!\!\!\!\!\!+2\int_0^t\int_{K_\varepsilon^c}\Big\langle f(s,X^\psi_s,z)\big(\psi_n(s,z)-\psi(s,z)\big),X^{\psi_n}_s-X^{\psi}_s\Big\rangle_H\nu(dz)ds\nonumber\\
:=&&\!\!\!\!\!\!\!\!I_{n,1}(t)+I_{n,2}(t).
\end{eqnarray}
By \textbf{(H6)} and \eqref{eq uniform Skelekon},
\begin{eqnarray}\label{In2}
&&\!\!\!\!\!\!\!\!\sup_{t\in[0,T]}|I_{n,2}(t)|\nonumber\\
\leq&&\!\!\!\!\!\!\!\! 2\int_0^T\int_{K_\varepsilon^c}L_f(s,z)\cdot(\|X^\psi_s\|_H+1)\cdot|\psi_n(s,z)-\psi(s,z)|\cdot\|X^{\psi_n}_s-X^\psi_s\|_H\nu(dz)ds\nonumber\\
\leq&&\!\!\!\!\!\!\!\!2\sup_{s\in[0,T]}\Big[(\|X^\psi_s\|_H+1)\cdot\|X^{\psi_n}_s-X^\psi_s\|_H\Big]\cdot\nonumber\\
&&\ \ \ \ \Big(\int_0^T\int_{K_\varepsilon^c}L_f(s,z)|\psi_n(s,z)-1|\nu(dz)ds+\int_0^T\int_{K_\varepsilon^c}L_f(s,z)|\psi(s,z)-1|\nu(dz)ds\Big)\nonumber\\
\leq&&\!\!\!\!\!\!\!\! \varepsilon C_N.
\end{eqnarray}
To estimate $I_{n,1}(t)$, define
$$A_{L_f,J}=\{(s,z)\in[0,T]\times Z: L_f(s,z)\geq J\}.$$
In the following, for any $M\subset [0,T]\times Z$, denote $M^c:=[0,T]\times Z\setminus M$. Denote
\begin{eqnarray}\label{In1}
&&\!\!\!\!\!\!\!\!I_{n,1}(t)\nonumber\\
=&&\!\!\!\!\!\!\!\!2\int_0^t\int_{K_\varepsilon}\Big\langle f(s,X^\psi_s,z)\big(\psi_n(s,z)-\psi(s,z)\big),X^{\psi_n}_s-X^{\psi}_s\Big\rangle_H1_{A_{L_f,J}}(s,z)\nu(dz)ds\nonumber\\
&&\!\!\!\!\!\!\!\!+2\int_0^t\int_{K_\varepsilon}\Big\langle f(s,X^\psi_s,z)\big(\psi_n(s,z)-\psi(s,z)\big),X^{\psi_n}_s-X^{\psi}_s\Big\rangle_H1_{A^c_{L_f,J}}(s,z)\nu(dz)ds\nonumber\\
:=&&\!\!\!\!\!\!\!\!I_{n,1,J}(t)+I_{n,1,J^c}(t).
\end{eqnarray}
By \textbf{(H6)} and \eqref{eq uniform Skelekon},
\begin{eqnarray}\label{In1J}
&&\!\!\!\!\!\!\!\!\sup_{t\in[0,T]}|I_{n,1,J}(t)|\nonumber\\
\leq&&\!\!\!\!\!\!\!\!2\int_0^T\int_{K_\varepsilon}L_f(s,z)(\|X^\psi_s\|_H+1)\big(\psi_n(s,z)+\psi(s,z)\big)\cdot\nonumber\\
&&\ \ \ \ \ \ \ \ \ \ \ \ \ \ \ (\|X^{\psi_n}_s\|_H+\|X^\psi_s\|_H)1_{A_{L_f,J}}(s,z)\nu(dz)ds\nonumber\\
\leq&&\!\!\!\!\!\!\!\!2\sup_{s\in[0,T]}\big[(\|X^\psi_s\|_H+1)(\|X^{\psi_n}_s\|_H+\|X^\psi_s\|_H)\big]\cdot\nonumber\\
&&\ \ \ \ \ \ \ \ \ \ \ \ \ \ \ \int_0^T\int_{K_\varepsilon}L_f(s,z)\big(\psi_n(s,z)+\psi(s,z)\big)1_{A_{L_f,J}}(s,z)\nu(dz)ds\nonumber\\
\leq&&\!\!\!\!\!\!\!\!C_N\sup_{\hbar\in S^N}\int_0^T\int_{K_\varepsilon}L_f(s,z)\hbar(s,z)1_{A_{L_f,J}}(s,z)\nu(dz)ds,
\end{eqnarray}
by \eqref{Lf}, we know that for $\varepsilon>0$, there exists $J_\varepsilon>0$ such that
\begin{eqnarray*}
\sup_{\hbar\in S^N}\int_0^T\int_{K_\varepsilon}L_f(s,z)\hbar(s,z)1_{A_{L_f,J_\varepsilon}}(s,z)\nu(dz)ds\leq\frac{\varepsilon}{C_N},
\end{eqnarray*}
so choose $J$ in \eqref{In1J} to be $J_\varepsilon$, then \eqref{In1J} yields
\begin{eqnarray}\label{In1j}
\sup_{t\in[0,T]}|I_{n,1,J_\varepsilon}(t)|\leq\varepsilon.
\end{eqnarray}
Substituting \eqref{Qn1}-\eqref{In1j} to \eqref{QN1}, we get
\begin{eqnarray}\label{condition1}
&&\!\!\!\!\!\!\!\!\sup_{t\in[0,T]}\|X^{\psi_n}_t-X^{\psi}_t\|_H^2\nonumber\\
\leq&&\!\!\!\!\!\!\!\!C_{G_f,N,\beta,\alpha,T}\cdot\Big(\varepsilon+\sup_{t\in[0,T]}|I_{n,1,J^c_\varepsilon}(t)|+\varepsilon C_N\Big).
\end{eqnarray}
To estimate $|I_{n,1,J_\varepsilon^c}(t)|$, denote
\begin{eqnarray*}
U^n(s)=X^{\psi_n}_s-X^\psi_s,~~~U^n(\bar{s}_m)=X^{\psi_n}_{\bar{s}_m}-X^\psi_{\bar{s}_m},
\end{eqnarray*}
where
\begin{eqnarray*}
\bar{s}_m=t_{k+1}\equiv(k+1)T\cdot2^{-m},~~\text{for}~~s\in[kT2^{-m}, (k+1)T2^{-m}).
\end{eqnarray*}
Then
\begin{eqnarray}\label{sum}
\sup_{t\in[0,T]}|I_{n,1,J^c_\varepsilon}(t)|\leq\sum_{i=1}^4\tilde{I_i},
\end{eqnarray}
where
\begin{eqnarray*}
&&\!\!\!\!\!\!\!\!\tilde{I_1}=\!\!\sup_{t\in[0,T]}\Big|\!\!\int_0^t\!\!\int_{K_\varepsilon}\!\!\big\langle\!f(s,X^\psi_s,z)\big(\psi_n(s,z)-\psi(s,z)\big),U^n(s)-U^n(\bar{s}_m)\big\rangle_{H}1_{A^c_{L_f,J_\varepsilon}}(s,z)\nu(dz)ds\Big|,\\
&&\!\!\!\!\!\!\!\!\tilde{I_2}=\!\!\sup_{t\in[0,T]}\Big|\!\!\int_0^t\!\!\int_{K_\varepsilon}\!\!\big\langle\!\big(f(s,X^\psi_s,z)-f(s,X^\psi_{\bar{s}_m},z)\big)\big(\psi_n(s,z)-\psi(s,z)\big),U^n(\bar{s}_m)\big\rangle_{H}1_{A^c_{L_f,J_\varepsilon}}(s,z)\nu(dz)ds\Big|,\\
&&\!\!\!\!\!\!\!\!\tilde{I_3}=\!\!\sup_{1\leq k\leq2^m}\!\sup_{t_{k-1}\leq t\leq t_k}\!\Big|\!\!\int_{t_{k-1}}^t\!\!\int_{K_\varepsilon}\!\!\big\langle\!f(s,X^\psi_{\bar{s}_m},z)\big(\psi_n(s,z)-\psi(s,z)\big),U^n(\bar{s}_m)\big\rangle_{H}\!1_{A^c_{L_f,J_\varepsilon}}(s,z)\nu(dz)ds\Big|,\\
&&\!\!\!\!\!\!\!\!\tilde{I_4}=\!\!\sum_{k=1}^{2^m}\Big|\!\!\int_{t_{k-1}}^{t_k}\!\!\int_{K_\varepsilon}\!\!\big\langle\!f(s,X^\psi_{\bar{s}_m},z)\big(\psi_n(s,z)-\psi(s,z)\big),U^n(\bar{s}_m)\big\rangle_{H}1_{A^c_{L_f,J_\varepsilon}}(s,z)\nu(dz)ds\Big|.
\end{eqnarray*}
Note that $\tilde{I_i}, i=1,...,4$, are all dependent on $n,m,\varepsilon$. For simplicity, we omit these parameters.

Now, let us estimate $\tilde{I_i}, i=1,2,3,4$. From \cite[Remark 3.3]{BCDSPA} we know that for any $a, b\in (0,+\infty)$ and $\sigma\in [1,+\infty)$,
\begin{eqnarray}\label{ab}
ab\leq e^{\sigma a}+\frac{1}{\sigma}(b\log b-b+1)=e^{\sigma a}+\frac{1}{\sigma}\ell(b).
\end{eqnarray}
Choosing $a=1$ and $b=\psi_n(s,z)$ or $\psi(s,z)$ in \eqref{ab}, using \eqref{eq uniform Skelekon}, similarly as to get \cite[(5.23)]{WZSIAM},
\begin{eqnarray}\label{I1}
\tilde{I_1}\leq &&\!\!\!\!\!\!\!\!C_{N,T}J_\varepsilon\cdot e^{\sigma}\nu(K_\varepsilon)\Big[\big(\int_0^T\|X^{\psi_n}_s-X^{\psi_n}_{\bar{s}_m}\|_{H}^2ds\big)^{\frac{1}{2}}+\big(\int_0^T\|X^{\psi}_s-X^{\psi}_{\bar{s}_m}\|_{H}^2ds\big)^{\frac{1}{2}}\Big]\nonumber\\
&&\!\!\!\!\!\!\!\!+J_\varepsilon \cdot \frac{C_{N,T}}{\sigma}.
\end{eqnarray}
Here $C_{N,T}$ only depends on $C_N$ appearing in \eqref{eq uniform Skelekon} and $T$.

To estimate $\int_0^T\|X^{\psi_n}_s-X^{\psi_n}_{\bar{s}_m}\|_{H}^2ds$, note that for any $s\in[0,T]$,
\begin{eqnarray*}
X^{\psi_n}_{\bar{s}_m}-X^{\psi_n}_s=\int_s^{\bar{s}_m}\mathcal{A}(t,X^{\psi_n}_t)dt+\int_s^{\bar{s}_m}\int_Zf(t,X^{\psi_n}_t,z)\big(\psi_n(t,z)-1\big)\nu(dz)dt.
\end{eqnarray*}
Applying the chain rule to $\|X^{\psi_n}_{\bar{s}_m}-X^{\psi_n}_s\|_{H}^2$, by \eqref{triple},
\begin{eqnarray*}
&&\!\!\!\!\!\!\!\!\|X^{\psi_n}_{\bar{s}_m}-X^{\psi_n}_s\|_{H}^2\nonumber\\
=&&\!\!\!\!\!\!\!\!2\int_s^{\bar{s}_m}~_{V^*}\big\langle\mathcal{A}(t,X^{\psi_n}_t),X^{\psi_n}_t-X^{\psi_n}_s\big\rangle_Vdt\nonumber\\
&&\!\!\!\!\!\!\!\!+2\int_s^{\bar{s}_m}\big\langle\int_Zf(t,X^{\psi_n}_t,z)\big(\psi_n(t,z)-1\big)\nu(dz),X^{\psi_n}_t-X^{\psi_n}_s\big\rangle_Hdt,~~~\forall s\in[0,T].
\end{eqnarray*}
Integrating the above equality over $[0,T]$ with respect to $s$, we get
\begin{eqnarray*}
&&\!\!\!\!\!\!\!\!\int_0^T\|X^{\psi_n}_{\bar{s}_m}-X^{\psi_n}_s\|_{H}^2ds\nonumber\\
=&&\!\!\!\!\!\!\!\!2\int_0^T\int_s^{\bar{s}_m}~_{V^*}\big\langle\mathcal{A}(t,X^{\psi_n}_t),X^{\psi_n}_t-X^{\psi_n}_s\big\rangle_Vdtds\nonumber\\
&&\!\!\!\!\!\!\!\!+2\int_0^T\int_s^{\bar{s}_m}\big\langle\int_Zf(t,X^{\psi_n}_t,z)\big(\psi_n(t,z)-1\big)\nu(dz),X^{\psi_n}_t-X^{\psi_n}_s\big\rangle_Hdtds.
\end{eqnarray*}
By Young's inequality, \textbf{(H4)} and Fubini's theorem, there exists a positive constant $C_\alpha\in(0,+\infty)$, which only depends on $\alpha$ and may change from line to line, such that
\begin{eqnarray*}
&&\!\!\!\!\!\!\!\!2\int_0^T\int_s^{\bar{s}_m}~_{V^*}\big\langle\mathcal{A}(t,X^{\psi_n}_t),X^{\psi_n}_t-X^{\psi_n}_s\big\rangle_Vdtds\nonumber\\
\leq&&\!\!\!\!\!\!\!\!2\int_0^T\int_s^{\bar{s}_m}\frac{\|\mathcal{A}(t,X^{\psi_n}_t)\|_{V^*}^{\frac{\alpha}{\alpha-1}}}{\frac{\alpha}{\alpha-1}}+\frac{\|X^{\psi_n}_t-X^{\psi_n}_s\|_V^\alpha}{\alpha}dtds\nonumber\\
\leq&&\!\!\!\!\!\!\!\!\frac{2(\alpha-1)}{\alpha}\int_0^T\int_s^{\bar{s}_m}\big(F_t+C\|X^{\psi_n}_t\|_V^\alpha\big)\big(1+\|X^{\psi_n}_t\|_H^\beta\big)dtds\nonumber\\
&&\!\!\!\!\!\!\!\!+C_\alpha\int_0^T\int_s^{\bar{s}_m}\|X^{\psi_n}_t\|_V^\alpha+\|X^{\psi_n}_s\|_V^\alpha~dtds\nonumber\\
=&&\!\!\!\!\!\!\!\!\frac{2(\alpha-1)}{\alpha}\int_0^T\int_s^{\bar{s}_m}F_t+F_t\|X^{\psi_n}_t\|_H^\beta+C\|X^{\psi_n}_t\|_V^\alpha+C\|X^{\psi_n}_t\|_V^\alpha\|X^{\psi_n}_t\|_H^\beta~dtds\nonumber\\
&&\!\!\!\!\!\!\!\!+C_\alpha\int_0^T\int_s^{\bar{s}_m}\|X^{\psi_n}_t\|_V^\alpha~dtds+C_\alpha\int_0^T\int_s^{\bar{s}_m}\|X^{\psi_n}_s\|_V^\alpha~dtds\nonumber\\
\leq&&\!\!\!\!\!\!\!\!C_\alpha\int_0^T\int_s^{\bar{s}_m}F_t+\|X^{\psi_n}_t\|_V^\alpha~dtds+C_\alpha\cdot\Big[\sup_{\hbar\in S^N}\sup_{t\in[0,T]}\|X^{\hbar}_t\|_H^\beta\Big]\cdot\int_0^T\int_s^{\bar{s}_m}F_t+\|X^{\psi_n}_t\|_V^\alpha~dtds\nonumber\\
&&\!\!\!\!\!\!\!\!+\frac{T}{2^m}C_\alpha\cdot\sup_{\hbar\in S^N}\int_0^T\|X^{\hbar}_s\|_V^\alpha ds\nonumber\\
=&&\!\!\!\!\!\!\!\!C_\alpha\int_0^T\int_{\bar{s}_{m-1}}^tF_t+\|X^{\psi_n}_t\|_V^\alpha~dsdt+C_\alpha\cdot\Big[\sup_{\hbar\in S^N}\sup_{t\in[0,T]}\|X^{\hbar}_t\|_H^\beta\Big]\cdot\int_0^T\int_{\bar{s}_{m-1}}^tF_t+\|X^{\psi_n}_t\|_V^\alpha~dsdt\nonumber\\
&&\!\!\!\!\!\!\!\!+\frac{T}{2^m}C_\alpha\cdot\sup_{\hbar\in S^N}\int_0^T\|X^{\hbar}_s\|_V^\alpha ds\nonumber\\
\leq&&\!\!\!\!\!\!\!\!\frac{T}{2^m}C_\alpha\cdot\Big[\int_0^TF_tdt+\sup_{\hbar\in S^N}\int_0^T\|X^{\hbar}_t\|_V^\alpha dt\Big]\nonumber\\
&&\!\!\!\!\!\!\!\!+\frac{T}{2^m}C_\alpha\cdot\Big[\sup_{\hbar\in S^N}\sup_{t\in[0,T]}\|X^{\hbar}_t\|_H^\beta\Big]\cdot\Big[\int_0^TF_tdt+\sup_{\hbar\in S^N}\int_0^T\|X^{\hbar}_t\|_V^\alpha dt\Big]\nonumber\\
&&\!\!\!\!\!\!\!\!+\frac{T}{2^m}C_\alpha\cdot\sup_{\hbar\in S^N}\int_0^T\|X^{\hbar}_s\|_V^\alpha ds.
\end{eqnarray*}
Since $F\in L^1([0,T];\Bbb{R}^+)$ and \eqref{eq uniform Skelekon} holds, the above inequality can be dominated by $\frac{1}{2^m} C_{\alpha,T,N}$, where $C_{\alpha,T,N}$ only depends on $C_N$ appearing in \eqref{eq uniform Skelekon}, $T$ and $\alpha$.

By \textbf{(H6)} and \eqref{eq uniform Skelekon},
\begin{eqnarray*}
&&\!\!\!\!\!\!\!\!\int_0^T\int_s^{\bar{s}_m}\big\langle\int_Zf(t,X^{\psi_n}_t,z)\big(\psi_n(t,z)-1\big)\nu(dz),X^{\psi_n}_t-X^{\psi_n}_s\big\rangle_Hdtds\nonumber\\
\leq&&\!\!\!\!\!\!\!\!\int_0^T\int_s^{\bar{s}_m}\int_ZL_f(t,z)\big(1+\|X^{\psi_n}_t\|_H\big)\|X^{\psi_n}_t-X^{\psi_n}_s\|_H\cdot|\psi_n(t,z)-1|\nu(dz)dtds\nonumber\\
\leq&&\!\!\!\!\!\!\!\!4\Big[\sup_{\hbar\in S^N}\sup_{t\in[0,T]}\|X^\hbar_t\|_H^2+1\Big]\int_0^T\int_s^{\bar{s}_m}\int_ZL_f(t,z)\cdot|\psi_n(t,z)-1|\nu(dz)dtds\nonumber\\
\leq&&\!\!\!\!\!\!\!\!C_{N,T}\sup_{\hbar\in S^N}\sup_{s\in[0,T]}\int_s^{\bar{s}_m}\int_ZL_f(t,z)\cdot|\hbar(t,z)-1|\nu(dz)dt.
\end{eqnarray*}
Since by \eqref{eq lemma-est-ass 03},
\begin{eqnarray*}
\lim_{m\rightarrow+\infty}\sup_{\hbar\in S^N}\sup_{s\in[0,T]}\int_s^{\bar{s}_m}\int_ZL_f(t,z)\cdot|\hbar(t,z)-1|\nu(dz)dt=0,
\end{eqnarray*}
we then have
\begin{eqnarray}\label{limpsin}
\lim_{m\rightarrow+\infty}\sup_{n\in\Bbb{N}}\int_0^T\|X^{\psi_n}_{\bar{s}_m}-X^{\psi_n}_s\|_H^2ds=0.
\end{eqnarray}
Using similar arguments, we also have
\begin{eqnarray}\label{limpsi}
\lim_{m\rightarrow+\infty}\int_0^T\|X^{\psi}_{\bar{s}_m}-X^{\psi}_s\|_H^2ds=0.
\end{eqnarray}
Substituting \eqref{limpsin} and \eqref{limpsi} into \eqref{I1}, we get
\begin{eqnarray*}
\limsup_{m\rightarrow+\infty}\sup_{n\in\Bbb{N}}\tilde{I_1}\leq  J_\varepsilon\cdot\frac{C_{N,T}}{\sigma}.
\end{eqnarray*}
Since $\sigma$ is arbitrary in $[1,+\infty)$, we obtain
\begin{eqnarray*}\label{I1lim}
\lim_{m\rightarrow+\infty}\sup_{n\in\Bbb{N}}\tilde{I_1}=0.
\end{eqnarray*}
Similarly as to get \cite[(5.37), (5.39)]{WZSIAM}, we have
\begin{eqnarray*}\label{I2lim}
\lim_{m\rightarrow+\infty}\sup_{n\in\Bbb{N}}\tilde{I_2}=0,
\end{eqnarray*}
and
\begin{eqnarray*}\label{I3lim}
\limsup_{m\rightarrow+\infty}\sup_{n\in\Bbb{N}}\tilde{I_3}=0.
\end{eqnarray*}
Hence, for any $\gamma>0$, there exists $m_\gamma>0$ such that for all $m\geq m_\gamma$,
\begin{eqnarray}\label{I123}
\sum_{i=1}^3\sup_{n\in\mathbb{N}}\tilde{I_{i}}\leq \gamma.
\end{eqnarray}

Similarly as to get \cite[(5.46)]{WZSIAM}, we have
\begin{eqnarray}\label{limf}
\lim_{n\rightarrow+\infty}\Big|\int_{t_{k-1}}^{t_k}\!\!\int_{K_\varepsilon}\langle f(s,X^\psi_{\bar{s}_m},z),U^n(\bar{s}_m)\rangle_{H}\big(\psi_n(s,z)-\psi(s,z)\big)1_{A^c_{L_f,J_\varepsilon}}(s,z)\nu(dz)ds\Big|=0.
\end{eqnarray}
For the fixed $\gamma$ and $m_\gamma$ as above, \eqref{limf} implies that
\begin{eqnarray}\label{I6}
\lim_{n\rightarrow+\infty}\tilde{I_4}=0.
\end{eqnarray}
Then, by \eqref{sum}, \eqref{I123} and \eqref{I6},
\begin{eqnarray*}
\limsup_{n\rightarrow+\infty}\sup_{t\in[0,T]}|I_{n,1,J^c_\varepsilon}(t)|\leq\gamma,
\end{eqnarray*}
which implies
\begin{eqnarray}\label{In1Jc}
\lim_{n\rightarrow+\infty}\sup_{t\in[0,T]}|I_{n,1,J^c_\varepsilon}(t)|=0.
\end{eqnarray}
since $\gamma$ is arbitrary in $(0,+\infty)$.

Now, taking \eqref{In1Jc} into \eqref{condition1}, we finally obtain
\begin{eqnarray*}
\limsup_{n\rightarrow+\infty}\sup_{t\in[0,T]}\|X^{\psi_n}_t-X^\psi_t\|_H^2\leq C_{G_f,N,\beta,\alpha,T}\cdot(\varepsilon+C_N\cdot \varepsilon),
\end{eqnarray*}
since $\varepsilon$ is arbitrary in $(0,+\infty)$, it follows that
\begin{eqnarray*}
\lim_{n\rightarrow+\infty}\sup_{t\in[0,T]}\|X^{\psi_n}_t-X^\psi_t\|_H^2=0,
\end{eqnarray*}
which completes the proof of Proposition \ref{Prop LDP1}.
\end{proof}

\ep

\section{The verification of \textbf{(LDP2)} }\label{LDP 2}
Recall from \eqref{sol-control} that
$X^{\psi_\epsilon}:=\Gamma^\epsilon( N^{\epsilon^{-1}\psi_\epsilon})$ is the unique solution of \eqref{sol-control 01}. Hence \textbf{(LDP2)} is verified once the following proposition is proved.

\begin{proposition}\label{Prop LDP2}For any given $N\in\mathbb{N}$, let $\{\psi_\epsilon,~\epsilon>0\}\subset \mathcal{S}^N$. Then, for any $\delta>0$,
  \begin{align}\label{prop-2-main-eq}
\lim_{\epsilon\rightarrow0}\mathbb{P}\left(\sup_{t\in[0,T]}\|X^{\psi_\epsilon}_t-{\Gamma}(\psi_{\epsilon})(t)\|_{H}\geq \delta\right)=0.
\end{align}
\end{proposition}

\begin{proof}
For any given $N\in\mathbb{N}$, let $\{\psi_\epsilon,~\epsilon>0\}\subset \mathcal{S}^N$. Recall from \eqref{def G0} that   $
Y^{\psi_\epsilon}:=\Gamma(\psi_\epsilon) \in C([0,T],H)\cap L^\alpha([0,T],V)$   is the unique solution of the following equation
\begin{eqnarray*}\label{eq Skeleton-pro}
  Y^{\psi_\epsilon}_t=x+\int_0^t\mathcal{A}(s,Y^{\psi_\epsilon}_s)\,ds+\int_0^t\int_Zf(s,Y^{\psi_\epsilon}_s,z)(\psi_\epsilon(s,z)-1)\nu(dz)ds.
\end{eqnarray*}
Taking \eqref{sol-control 01} into account, we get
\begin{eqnarray*}
X^{\psi_\epsilon}_t-  Y^{\psi_\epsilon}_t
=&&\!\!\!\!\!\!\!\! \int_0^t\mathcal{A}(s,X^{\psi_\epsilon}_s)-\mathcal{A}(s,Y^{\psi_\epsilon}_s)\, ds+
  \epsilon\int_0^t\int_{ Z}f(s,X^{\psi_\epsilon}_{s-},z)\widetilde{N}^{\epsilon^{-1}\psi_\epsilon}(dz,ds)\\
  &&\!\!\!\!\!\!\!\!+\int_0^t\int_{ Z}\big(f(s,X^{\psi_\epsilon}_{s},z)-f(s,Y^{\psi_\epsilon}_s,z)\big)\big(\psi_\epsilon(s,z)-1\big)\nu(dz)ds
  ,\ \ \forall t\in[0,T].\nonumber
\end{eqnarray*}
By applying  the It\^o formula, we infer
\begin{eqnarray*}
\|X^{\psi_\epsilon}_t-  Y^{\psi_\epsilon}_t\|_H^2\nonumber
 =&&\!\!\!\!\!\!\!\!2\int_0^t ~_{V^*}\langle \mathcal{A}(s,X^{\psi_\epsilon}_s)-\mathcal{A}(s,Y^{\psi_\epsilon}_s), X^{\psi_\epsilon}_s - Y^{\psi_\epsilon}_s\rangle_V\;ds \\
 &&\!\!\!\!\!\!\!\! +2\epsilon \int_0^t\int_{ Z} \langle  f(s,X^{\psi_\epsilon}_{s-},z),  X^{\psi_\epsilon}_{s-}-  Y^{\psi_\epsilon}_{s-}\rangle_H\; \widetilde{N}^{\epsilon^{-1}\psi_\epsilon}(dz,ds)\\
 &&\!\!\!\!\!\!\!\!+ 2\int_0^t\int_{ Z}\langle\big(f(s,X^{\psi_\epsilon}_{s},z)-f(s,Y^{\psi_\epsilon}_s,z)\big)\big(\psi_\epsilon(s,z)-1\big) ,X^{\psi_\epsilon}_s-  Y^{\psi_\epsilon}_s\rangle_H\; \nu(dz)ds\\
 &&\!\!\!\!\!\!\!\!+\epsilon^2 \int_0^t\int_{ Z}\|f(s,X^{\psi_\epsilon}_{s-},z)\|_H^2\; N^{\epsilon^{-1}\psi_\epsilon}(dz,ds),\ \ \forall t\in[0,T].
\end{eqnarray*}
Owing to {\bf{(H2)}} and {\bf {(H7)}}, we deduce
\begin{eqnarray}\label{prop-2-eq1}
\|X^{\psi_\epsilon}_t-  Y^{\psi_\epsilon}_t\|_H^2\nonumber
 \leq&&\!\!\!\!\!\!\!\! \int_0^t (F_s+\rho(Y^{\psi_\epsilon}_s ))\| X^{\psi_\epsilon}_s-Y^{\psi_\epsilon}_s\|^2_H\; ds\nonumber\\
 &&\!\!\!\!\!\!\!\! +2\epsilon \int_0^t\int_{ Z} \langle  f(s,X^{\psi_\epsilon}_{s-},z),  X^{\psi_\epsilon}_{s-}-  Y^{\psi_\epsilon}_{s-}\rangle_H\; \widetilde{N}^{\epsilon^{-1}\psi_\epsilon}(dz,ds)\nonumber\\
 &&\!\!\!\!\!\!\!\!+2\int_0^t\int_Z G_f(s,z)|\psi_\epsilon(s,z) -1|\cdot\|X^{\psi_\epsilon}_{s}-  Y^{\psi_\epsilon}_{s} \|^2_H \nu(dz)ds\nonumber\\
&&\!\!\!\!\!\!\!\!+\epsilon^2 \int_0^t\int_{ Z}\|f(s,X^{\psi_\epsilon}_{s-},z)\|_H^2\; N^{\epsilon^{-1}\psi_\epsilon}(dz,ds)\nonumber\\
 \leq&&\!\!\!\!\!\!\!\!2 \int_0^t (G_\epsilon(s)+F_s+\rho(Y^{\psi_\epsilon}_s ))\| X^{\psi_\epsilon}_s-Y^{\psi_\epsilon}_s\|^2_H\; ds\nonumber\\
 &&\!\!\!\!\!\!\!\!+2\epsilon \int_0^t\int_{ Z} \langle  f(s,X^{\psi_\epsilon}_{s-},z),  X^{\psi_\epsilon}_{s-}-  Y^{\psi_\epsilon}_{s-}\rangle_H\; \widetilde{N}^{\epsilon^{-1}\psi_\epsilon}(dz,ds)\nonumber\\
   &&\!\!\!\!\!\!\!\!+\epsilon^2 \int_0^t\int_{ Z}\|f(s,X^{\psi_\epsilon}_{s-},z)\|_H^2\; N^{\epsilon^{-1}\psi_\epsilon}(dz,ds),\ \ \forall t\in[0,T],
\end{eqnarray}
where $G_\epsilon(s):=\int_Z G_f(s,z)|\psi_\epsilon(s,z) -1| \nu(dz)$. From \eqref{eq lemma-est-ass 02} there exists a constant $C_{G_f}$ such that, for any $\epsilon\in(0,1]$,
\begin{eqnarray}\label{eq S 20230918} \int_0^TG_\epsilon(s) ds\leq C_{G_f}<+\infty, \ \mathbb{P}\text{-a.s.}.
\end{eqnarray}

On applying Gronwall's inequality, we find
\begin{eqnarray}\label{prop-2-eq1-20231008}
&&\!\!\!\!\!\!\!\!\sup_{t\in[0,T]}\|X^{\psi_\epsilon}_t-  Y^{\psi_\epsilon}_t\|_H^2\nonumber\\
 \leq&&\!\!\!\!\!\!\!\! \Big(2\epsilon \sup_{t\in[0,T]}\Big|\int_0^t\int_{ Z} \langle  f(s,X^{\psi_\epsilon}_{s-},z),  X^{\psi_\epsilon}_{s-}-  Y^{\psi_\epsilon}_{s-}\rangle_H\; \widetilde{N}^{\epsilon^{-1}\psi_\epsilon}(dz,ds)\Big|\nonumber\\
   &&\!\!\!\!\!\!\!\!\ \ \ \ \ \ \ \ \ \ \ \ \ +\epsilon^2 \int_0^T\int_{ Z}\|f(s,X^{\psi_\epsilon}_{s-},z)\|_H^2\; N^{\epsilon^{-1}\psi_\epsilon}(dz,ds)\Big)\cdot e^{2 \int_0^T G_\epsilon(s)+F_s+\rho(Y^{\psi_\epsilon}_s )ds}
   \nonumber\\
   \leq&&\!\!\!\!\!\!\!\!
   C_{N,\int_0^TF_sds,C_{G_f}}
  \Big(2\epsilon \sup_{t\in[0,T]}\Big|\int_0^t\int_{ Z} \langle  f(s,X^{\psi_\epsilon}_{s-},z),  X^{\psi_\epsilon}_{s-}-  Y^{\psi_\epsilon}_{s-}\rangle_H\; \widetilde{N}^{\epsilon^{-1}\psi_\epsilon}(dz,ds)\Big|\nonumber\\
   &&\ \ \ \ \ \ \ \ \ \ \ \ \ \ \ \ \ \ \ \ \ \ \ \ \ \ +\epsilon^2 \int_0^T\int_{ Z}\|f(s,X^{\psi_\epsilon}_{s-},z)\|_H^2\; N^{\epsilon^{-1}\psi_\epsilon}(dz,ds)\Big),\ \mathbb{P}\text{-a.s.}.
\end{eqnarray}
in which \eqref{eq S 20230918}, \eqref{eq uniform Skelekon} and {\bf (H5)} were used to get
$$
2 \int_0^T G_\epsilon(s)+F_s+\rho(Y^{\psi_\epsilon}_s )ds
\leq
C_{N,\int_0^TF_sds,C_{G_f}}<+\infty, \ \mathbb{P}\text{-a.s.},
$$
where $C_{N,\int_0^TF_sds,C_{G_f}}$ is a constant, and only depends on $C_N$ appearing in \eqref{eq uniform Skelekon},  $\int_0^TF_sds$ and $C_{G_f}$.

Applying Doob's inequality for $p=1$ (c.f. \cite[Theorem 1]{Novikov} or \cite[Proposition 2.2]{ZBL}) in the second term on the right-hand side of \eqref{prop-2-eq1} gives
\begin{eqnarray}\label{prop-2-eq2}
&&\!\!\!\!\!\!\!\!2\epsilon\,  \mathbb{E}\sup_{t\in[0,T]}\Big| \int_0^t\int_{ Z} \langle  f(s,X^{\psi_\epsilon}_{s-},z),  X^{\psi_\epsilon}_{s-}-  Y^{\psi_\epsilon}_{s-}\rangle_H\; \widetilde{N}^{\epsilon^{-1}\psi_\epsilon}(dz,ds)\Big|\nonumber\\
  \leq&&\!\!\!\!\!\!\!\! 2\epsilon\,  \mathbb{E}\Big[     \int_0^T\int_{ Z} \langle  f(s,X^{\psi_\epsilon}_{s},z),  X^{\psi_\epsilon}_{s}-  Y^{\psi_\epsilon}_{s}\rangle_H^2\;   \epsilon^{-1} \psi_\epsilon(s,z) \nu(dz)ds  \Big]^{\frac12}\nonumber\\
  \leq&&\!\!\!\!\!\!\!\!  2\epsilon^{\frac12}\,\mathbb{E}\Big[     \int_0^T\int_{ Z}2|L_f(s,z)|^2(1+ \|X^{\psi_\epsilon}_{s}\|_H^2)  \|X^{\psi_\epsilon}_{s}-  Y^{\psi_\epsilon}_{s}\|^2_H   \psi_\epsilon(s,z) \nu(dz)ds       \Big]^{\frac12}\nonumber\\
  \leq&&\!\!\!\!\!\!\!\!   2\epsilon^{\frac12}\,\mathbb{E}\Big[  \sup_{s\in[0,T]}\|X^{\psi_\epsilon}_{s}-  Y^{\psi_\epsilon}_{s}\|^2_H   \int_0^T\int_{ Z} 2|L_f(s,z)|^2(1+ \|X^{\psi_\epsilon}_{s}\|_H^2) \psi_\epsilon(s,z) \nu(dz)ds       \Big]^{\frac12}\nonumber \\
  \leq&&\!\!\!\!\!\!\!\!   \epsilon^{\frac12}\,\mathbb{E}\Big[  \sup_{s\in[0,T]}\|X^{\psi_\epsilon}_{s}-  Y^{\psi_\epsilon}_{s}\|^2_H    \Big]+  \epsilon^{\frac12}\,\mathbb{E}\Big[  \int_0^T\int_{ Z} 2|L_f(s,z)|^2(1+ \|X^{\psi_\epsilon}_{s}\|_H^2)\psi_\epsilon(s,z) \nu(dz)ds       \Big]\nonumber\\
  \leq&&\!\!\!\!\!\!\!\!   \epsilon^{\frac12}\,\mathbb{E}\Big[  \sup_{s\in[0,T]}\|X^{\psi_\epsilon}_{s}-  Y^{\psi_\epsilon}_{s}\|^2_H    \Big]+2  \epsilon^{\frac12}\mathbb{E}\Big[     \sup_{s\in[0,T]} (1+ \|X^{\psi_\epsilon}_{s}\|_H^2) \int_0^T\int_{ Z}  |L_f(s,z) |^2\psi_\epsilon(s,z)\, \nu(dz)ds  \Big]\nonumber\\
  \leq&&\!\!\!\!\!\!\!\! \epsilon^{\frac12}\,\mathbb{E}\Big[  \sup_{s\in[0,T]}\|X^{\psi_\epsilon}_{s}-  Y^{\psi_\epsilon}_{s}\|^2_H    \Big]+2 \epsilon^{\frac12} C_{L_f,N} \,\mathbb{E}\Big[     \sup_{s\in[0,T]} (1+ \|X^{\psi_\epsilon}_{s}\|_H^2)  \Big],
\end{eqnarray}
where $C_{L_f,N}:=\sup _{\hbar \in S^{N}}\int_0^T\int_{Z}|L_f(t,z)|^{2}(\hbar(t, z)+1) \nu(\dif z) \dif t<+\infty  $ by \eqref{eq lemma-est-ass 01} and we also used Assumption  {\bf{(H6)}} in the second inequality and Young's inequality  in the fourth inequality.

For the last term on the right-hand side of \eqref{prop-2-eq1}, using again Assumption {\bf{(H6)}}, we find
\begin{eqnarray}\label{prop-2-eq3}
&&\!\!\!\!\!\!\!\! \epsilon^2 \mathbb{E}\Big[ \int_0^T\int_{ Z}\|f(s,X^{\psi_\epsilon}_{s-},z)\|_H^2\; N^{\epsilon^{-1}\psi_\epsilon}(dz,ds)\Big]\nonumber\\
=&&\!\!\!\!\!\!\!\! \epsilon\,  \mathbb{E}\Big[ \int_0^T\int_{ Z}\|f(s,X^{\psi_\epsilon}_{s},z)\|_H^2\,\psi_\epsilon(s,z)\; \nu(dz)ds\Big]\nonumber\\
\leq&&\!\!\!\!\!\!\!\!   \epsilon\,  \mathbb{E}\Big[ \int_0^T\int_{ Z}2 |L_f(s,z)|^2(1+\|X^{\psi_\epsilon}_s\|_H^2)\,\psi_\epsilon(s,z)\; \nu(dz)ds\Big]\nonumber\\
\leq&&\!\!\!\!\!\!\!\!  2 \epsilon\, C_{L_f,N} \,\mathbb{E}\Big[     \sup_{s\in[0,T]} (1+ \|X^{\psi_\epsilon}_{s}\|_H^2)  \Big].
\end{eqnarray}
Combining \eqref{prop-2-eq2}, \eqref{prop-2-eq3} with \eqref{prop-2-eq1-20231008}, we obtain for any $\epsilon \in(0,1]$,
\begin{eqnarray}\label{LDP2 ESTIMATE}
&&\!\!\!\!\!\!\!\! \mathbb{E}\Big[\sup_{s\in[0,T]}  \|X^{\psi_\epsilon}_s-  Y^{\psi_\epsilon}_s\|_H^2\Big]\nonumber\\
\leq&&\!\!\!\!\!\!\!\! \epsilon^{\frac12}C_{N,\int_0^TF_sds,C_{G_f}}\,\mathbb{E}\Big[  \sup_{s\in[0,T]}\|X^{\psi_\epsilon}_{s}-  Y^{\psi_\epsilon}_{s}\|^2_H    \Big]\nonumber\\
&& \!\!\!\!\!\!\!\! +2 (\epsilon^{\frac12}+\epsilon)\, C_{L_f,\int_0^TF_sds,N,C_{G_f}} \,\mathbb{E}\Big[     \sup_{s\in[0,T]} (1+ \|X^{\psi_\epsilon}_{s}\|_H^2)  \Big],
\end{eqnarray}
where $C_{L_f,\int_0^TF_sds,N,C_{G_f}}$ is a constant, and only depends on $C_N$ appearing in \eqref{eq uniform Skelekon}, $L_f$, $\int_0^TF_sds$ and $C_{G_f}$.

By applying the It\^o formula to $\|X^{\psi_\epsilon}_t\|^2_H$, and using Assumptions {\bf{(H3)}} and {\bf{(H6)}}, we infer
\begin{eqnarray*}
&&\!\!\!\!\!\!\!\! \|X^{\psi_\epsilon}_t\|_H^2\nonumber\\
=&&\!\!\!\!\!\!\!\! \|x\|_H^2+2\int_0^t ~_{V^*}\langle \mathcal{A}(s,X^{\psi_\epsilon}_s), X^{\psi_\epsilon}_s \rangle_V\;ds+2\epsilon \int_0^t\int_{ Z} \langle  f(s,X^{\psi_\epsilon}_{s-},z),  X^{\psi_\epsilon}_{s-}\rangle_H\; \widetilde{N}^{\epsilon^{-1}\psi_\epsilon}(dz,ds)\nonumber\\
 &&\!\!\!\!\!\!\!\! + 2\int_0^t\int_{ Z}\langle f(s,X^{\psi_\epsilon}_{s},z)\big(\psi_\epsilon(s,z)-1\big) ,X^{\psi_\epsilon}_s\rangle_H\; \nu(dz)ds\nonumber\\
 &&\!\!\!\!\!\!\!\! +\epsilon^2 \int_0^t\int_{ Z}\|f(s,X^{\psi_\epsilon}_{s-},z)\|_H^2\; N^{\epsilon^{-1}\psi_\epsilon}(dz,ds)\nonumber\\
\leq &&  \!\!\!\!\!\!\!\!
 \|x\|_H^2+ \int_0^t\big(F_s+4L_\epsilon(s)\big)(1+\|X^{\psi_\epsilon}_s\|_H^2)ds+\epsilon^2 \int_0^t\int_{ Z}\|f(s,X^{\psi_\epsilon}_{s-},z)\|_H^2\; N^{\epsilon^{-1}\psi_\epsilon}(dz,ds)\nonumber\\
 &&\!\!\!\!\!\!\!\!  +2\epsilon \int_0^t\int_{ Z} \langle  f(s,X^{\psi_\epsilon}_{s-},z),  X^{\psi_\epsilon}_{s-}\rangle_H\; \widetilde{N}^{\epsilon^{-1}\psi_\epsilon}(dz,ds),,\ \ \forall t\in[0,T],
\end{eqnarray*}
where $L_\epsilon(t):=\int_ZL_f(t,z)\cdot|\psi_\epsilon(s,z)-1|\nu(dz)$. From \eqref{eq lemma-est-ass 02} there exists a constant $C_{L_f}$ such that, for any $\epsilon\in(0,1]$,
\begin{eqnarray*}
\int_0^TL_\epsilon(s) ds\leq C_{L_f}<+\infty,\ \mathbb{P}\text{-a.s.}.
\end{eqnarray*}
Gronwall's inequality and the above inequality now provide
\begin{eqnarray}\label{LDP2estimate}
&&\!\!\!\!\!\!\!\!\sup_{t\in[0,T]}\|X^{\psi_\epsilon}_t\|_H^2\nonumber\\
\leq && \!\!\!\!\!\!\!\!
\Big(\|x\|_H^2+\int_0^TF_sds+4C_{L_f}
+2\epsilon \sup_{t\in[0,T]}\Big|\int_0^t\int_{ Z} \langle  f(s,X^{\psi_\epsilon}_{s-},z),  X^{\psi_\epsilon}_{s-}\rangle_H\; \widetilde{N}^{\epsilon^{-1}\psi_\epsilon}(dz,ds)\Big|\nonumber\\
&&\!\!\!\!\!\!\!\!
\ \ \ \ \ \ \ \ \ \ \ \ \ \ \ \  +
 \epsilon^2 \int_0^T\int_{ Z}\|f(s,X^{\psi_\epsilon}_{s-},z)\|_H^2\; N^{\epsilon^{-1}\psi_\epsilon}(dz,ds)\Big)\cdot e^{\int_0^TF_sds+4C_{L_f}},\ \mathbb{P}\text{-a.s.}.
\end{eqnarray}

Similar to \eqref{prop-2-eq2}, in view of Doob's inequality  for $p=1$ and Assumption {\bf{(H6)}}, we obtain
\begin{eqnarray}
\label{prop-2-eq8}
&&\!\!\!\!\!\!\!\!2\epsilon\,\mathbb{E}\Big[\sup_{t\in[0,T]}\Big| \int_0^t\int_{ Z} \langle  f(s,X^{\psi_\epsilon}_{s-},z),  X^{\psi_\epsilon}_{s-}\rangle_H\; \widetilde{N}^{\epsilon^{-1}\psi_\epsilon}(dz,ds)\Big|\Big]\nonumber\\
\leq&&\!\!\!\!\!\!\!\!  2\epsilon^{\frac12}\,\mathbb{E}\Big[     \int_0^T\int_{ Z}2|L_f(s,z)|^2(1+ \|X^{\psi_\epsilon}_{s}\|_H^2)  \|X^{\psi_\epsilon}_{s}\|^2_H   \psi_\epsilon(s,z) \nu(dz)ds       \Big]^{\frac12}\nonumber\\
\leq &&\!\!\!\!\!\!\!\!4 \epsilon^{\frac12}\,C_{L_f,N}^{\frac12}\mathbb{E}\Big[\sup_{s\in[0,T]}(1+\|X^{\psi_\epsilon}_{s}\|^2_H)  \Big].
 \end{eqnarray}
Taking the estimates \eqref{prop-2-eq3} and \eqref{prop-2-eq8} into account, \eqref{LDP2estimate} yields
\begin{eqnarray*}
\mathbb{E}\Big[\sup_{s\in[0,T]}   \|X^{\psi_\epsilon}_s\|_H^2\Big]
\leq && \!\!\!\!\!\!\!\!
\Big(\|x\|_H^2+\int_0^TF_sds+4C_{L_f}
+4 \epsilon^{\frac12}\,C_{L_f,N}^{\frac12}\mathbb{E}\Big[\sup_{s\in[0,T]}(1+\|X^{\psi_\epsilon}_{s}\|^2_H)  \Big]\nonumber\\
&&\!\!\!\!\!\!\!\!  +2 \epsilon\, C_{L_f,N} \,\mathbb{E}\Big[     \sup_{s\in[0,T]} (1+ \|X^{\psi_\epsilon}_{s}\|_H^2)  \Big]
\Big)\cdot e^{\int_0^TF_sds+4C_{L_f}}.
 \end{eqnarray*}
 Let $\epsilon_0\in(0,1)$ be such that $(4\epsilon_0^{\frac12}\,C_{L_f,N}^{\frac12}+2 \epsilon_0\, C_{L_f,N})e^{\int_0^TF_sds+4C_{L_f}}\leq \frac12$. Then for any $\epsilon\in (0, \epsilon_0]$, we conclude
 \begin{eqnarray}\label{LDP2 ESTIMATE1}
\frac12 \mathbb{E}\Big[\sup_{s\in[0,T]}  \|X^{\psi_\epsilon}_s\|_H^2\Big] \leq &&\!\!\!\!\!\!\!\!\Big(\|x\|_H^2+\int_0^TF_sds+4C_{L_f}
\Big)\cdot e^{\int_0^TF_sds+4C_{L_f}}+\frac12\nonumber\\
:=&&\!\!\!\!\!\!\!\!C_{\|x\|_H,\int_0^TF_sds,C_{L_f}}<+\infty.
  \end{eqnarray}
Taking \eqref{LDP2 ESTIMATE1} into \eqref{LDP2 ESTIMATE}, we get that for any $\epsilon\in (0, \epsilon_0]$,
\begin{eqnarray*} \mathbb{E}\Big[\sup_{s\in[0,T]}  \|X^{\psi_\epsilon}_s-  Y^{\psi_\epsilon}_s\|_H^2\Big]
\leq&& \!\!\!\!\!\!\!\!\epsilon^{\frac12}C_{N,\int_0^TF_sds,C_{G_f}}\,\mathbb{E}\Big[  \sup_{s\in[0,T]}\|X^{\psi_\epsilon}_{s}-  Y^{\psi_\epsilon}_{s}\|^2_H    \Big]\nonumber\\
&&\!\!\!\!\!\!\!\!  +2 (\epsilon^{\frac12}+\epsilon)(2C_{\|x\|_H,\int_0^TF_sds,C_{L_f}}+1)\, C_{L_f,\int_0^TF_sds,N,C_{G_f}}.
\end{eqnarray*}
Let $\epsilon_1\in (0, \epsilon_0]$ be such that $\epsilon_1^{\frac12}C_{N,\int_0^TF_sds,C_{G_f}}<\frac{1}{2}$. Then for any $\epsilon\in(0,\epsilon_1]$,
\begin{eqnarray*} \mathbb{E}\Big[\sup_{s\in[0,T]}  \|X^{\psi_\epsilon}_s-  Y^{\psi_\epsilon}_s\|_H^2\Big]
\leq 4 (\epsilon^{\frac12}+\epsilon)(2C_{\|x\|_H,\int_0^TF_sds,C_{L_f}}+1)\, C_{L_f,\int_0^TF_sds,N,C_{G_f}},
\end{eqnarray*}
 which implies \eqref{prop-2-main-eq}.

 The proof of Proposition \ref{Prop LDP2} is complete.
\end{proof}

\end{document}